\documentclass[12pt,twoside]{amsart}
\usepackage{amsmath,amsthm,amssymb}
\usepackage{mathrsfs, array,tikz}
\usetikzlibrary{arrows}
\usetikzlibrary{decorations.markings}

\usepackage{caption,subcaption}
\usepackage{bbold}
\usepackage{url}
\usepackage{amsaddr}
\usepackage{xcolor}
\usepackage{soul}

\usetikzlibrary{matrix,positioning,decorations.pathreplacing}

\newtheorem{thm}{Theorem}[section]
\newtheorem{cor}[thm]{Corollary}

\newtheorem{lemma}[thm]{Lemma}

\newtheorem{prop}[thm]{Proposition}
\newtheorem{rmk}[thm]{Remark}
\newtheorem{obs}[thm]{Observation}
\newtheorem{example}[thm]{Example}

\theoremstyle{definition}
\newtheorem{defi}[thm]{Definition}

\DeclareMathOperator{\Gr}{Gr}

\DeclareMathOperator{\cl}{cl}

\DeclareMathOperator{\convex}{convex}

\begin{document}

\title{Positroid envelopes and graphic positroids}
\author{Jeremy Quail and Puck Rombach}
\address{Dept.\:of Mathematics \& Statistics \\ University of Vermont \\ Burlington, VT, USA}
\email{$\{\mbox{jeremy.quail}, \mbox{puck.rombach}\}$@uvm.edu}
\date{January 26, 2025}

\maketitle

\begin{abstract}
Positroids are matroids realizable by real matrices with all nonnegative maximal minors. They partition the ordered matroids into equivalence classes, called positroid envelope classes, by their Grassmann necklaces. We give an explicit graph construction that shows that every positroid envelope class contains a graphic matroid. We {prove} that a graphic positroid is the unique matroid in its positroid envelope class. Finally, we show that every graphic positroid has an oriented graph representable by a signed incidence matrix with all nonnegative minors.
\end{abstract}

\section{Introduction}

The Grassmannian $\Gr(k,\mathbb{F}^n)$ admits a stratification by the realization spaces of the rank-$k$ ordered matroids on $n$ elements. The study of this matroid stratification of the Grassmannian has proven difficult as, due to Mn\"ev's Universality Theorem, a matroid stratum can be as pathological as any algebraic variety. The totally nonnegative Grassmannian $\Gr^{\geq 0}(k,\mathbb{R}^n)$, consisting of the points in $\Gr(k,\mathbb{R}^n)$ that can be realized by matrices with all nonnegative maximal minors, has a stratification due to Postnikov into positroid cells that are each homeomorphic to an open ball~\cite{postnikov2006total}. 

Positroids are in bijection with a number of interesting combinatorial objects~\cite{postnikov2006total}, and have connections to scattering amplitudes in $\mathcal{N}=4$ supersymmetric Yang-Mills theory~\cite{arkani2012scattering}, tilings of the amplituhedron~\cite{parisi2021m}, and cluster algebras~\cite{scott2006grassmannians}. 
Postnikov's positroid stratification of the nonnegative Grassmannian yields a partition of the ordered matroids into classes with positroid representatives, {where these representatives} are called the positroid envelopes~\cite{knutson2013positroid}. We study the interplay between the family of graphic matroids and these positroid envelopes. We show that every positroid is the envelope of a graphic matroid. We {prove that} $P$ is a graphic positroid {if and only if} $P$ is the unique matroid contained in its positroid envelope class. 

The rest of this paper is organized as follows. In Section~\ref{sec:prelims}, we provide the necessary definitions and tools. In Section~\ref{sec:graph-constr}, we prove that every positroid envelope class contains at least one graphic matroid, by explicitly constructing a planar graph whose cycle matroid lies in the desired positroid envelope {class}. In Section~\ref{sec:graphic-pos}, we characterize the graphic positroids and their envelope classes. In Section~\ref{sec:incidence}, we strengthen the connection between the graphic positroids and their graphs, by showing that every graphic positroid $P$ can be realized by a matrix with all nonnegative maximal minors that is the signed incidence matrix of a graph that represents $P$.

\section{Basic tools and definitions}\label{sec:prelims}
In this section, we establish basic definitions and tools that are used throughout the paper. In~\ref{sec:matroidbasics}, we establish the basics for matroids. In~\ref{sec:positroidbasics}, we give necessary background on positroids. In~\ref{sec:decpermbasics}, we describe some combinatorial objects that are in bijection to the positroids: decorated permutations and Grassmann necklaces. These are used in the proofs of main results. In~\ref{sec:basics-envelopes}, we give the basic definitions of positroid envelopes and varieties.

\subsection{Matroids}\label{sec:matroidbasics}
A \emph{matroid} $M = (E,\mathcal{B})$ consists of a finite ground-set $E$ and a non-empty collection of subsets $\mathcal{B} \subset 2^E$, called bases, that satisfy the following \emph{basis exchange property}: 
\begin{itemize}
    \item if $B_1,B_2 \in \mathcal{B}$ and $x \in B_1 \setminus B_2$, then there exists $y \in B_2 \setminus B_1$ such that {$(B_1 \setminus \{x\}) \cup \{y\} \in \mathcal{B}$}.
\end{itemize}
The collection of \emph{independent sets} of $M$ is $\mathcal{I}(M) := \{ I \subseteq B: B \in \mathcal{B}(M) \}$, and the \emph{dependent sets} of $M$ are all subsets of $E(M)$ that are not independent. Minimal dependent sets of $M$ are called \emph{circuits}, {and we will denote the set of circuits of $M$ by $\mathcal{C}(M)$.} {An element of $M$ is called a \emph{loop} if it comprises a circuit consisting of a single element.} The \emph{dual} of a matroid $M$ is the matroid $M^*$ whose ground-set is $E(M^*) := E(M)$ and whose bases are {$\mathcal{B}(M^*) := \{ E(M) \setminus B : B \in \mathcal{B}(M) \}$}. A \emph{cocircuit} of $M$ is a circuit in $M^*$, and a \emph{coloop} of $M$ is a loop in $M^*$. The \emph{rank function} of a matroid $M$ is a map $r_M: 2^{E(M)} \to \mathbb{Z}_{\geq 0}$ given by, for all $X \subseteq E(M)$, $r_M(X) := \max \{ \lvert I \rvert : I \subseteq X, I \in \mathcal{I}(M) \}$. The \emph{closure operator} of a matroid $M$ is a function $\cl : 2^E \to 2^E$ given by, for all $X \subseteq E(M)$, $\cl(X) := \{e \in E(M) : r_M(X \cup e) = r_M(X) \}$. A \emph{matroid isomorphism} between two matroids $M$ and $N$ is a basis-preserving bijection {between $E(M)$ and $E(N)$}. 

Let $M$ be a matroid and $x \in E(M)$. The matroid obtained by \emph{deleting $x$ from $M$}, {denoted} $M \setminus x$, has ground-set $E(M \setminus x) := E(M) \setminus \{x\}$ and independent sets $\mathcal{I}(M \setminus x) := \{ I : x \notin I \in \mathcal{I}(M) \}$. The matroid obtained by \emph{contracting $x$ in $M$}, $M/x$, has ground-set $E(M/x) := E(M) \setminus \{x\}$ and bases
\begin{equation*}
    \mathcal{B}(M/x) :=
    \begin{cases}
        \{ B : B \in \mathcal{B}(M) \}, & \text{if } x \text{ is a loop}\\
        \{ B \setminus \{x\} : x \in B \in \mathcal{B}(M) \}, & \text{otherwise.}
    \end{cases}
\end{equation*}
We use the corresponding notation $(M/X)\setminus Y$, for $X,Y\subseteq E(M)$ and $X\cap Y =\emptyset$, to mean the contraction of the set of elements $X$ and {the} deletion of the set of elements $Y$. Deletion and contraction are dual to each other, in the sense that $M \setminus x = (M^*/x)^*$ and $M/x = (M^* \setminus x)^*$. We say that a matroid $N$ is a \emph{minor} of $M$ if it can be obtained from $M$ by a sequence of deletions and contractions. We say that a matroid $M$ is \emph{$N$-free} if $N$ is not {isomorphic to} a minor of $M$. {A class of matroids $\mathcal{F}$ is called \emph{minor-closed} if for any matroid $M$ in $\mathcal{F}$ every minor of $M$ is also contained in $\mathcal{F}$. Any minor-closed family $\mathcal{F}$ can be characterized by a set $S$, called the \emph{forbidden minors of $\mathcal{F}$}, of matroids not contained in $\mathcal{F}$ such that no matroid in $\mathcal{F}$ has a matroid in $S$ as a minor.}

Let $A$ be a matrix with entries over a field $\mathbb{F}$. We can obtain a matroid, {denoted} $M(A)$, from {the matrix} $A$ by taking the ground-set to be the column vectors of $A$, and the bases to be the collection of maximal linearly independent subsets of the column vectors of $A$. We say that $A$ is a matrix that realizes the matroid $M(A)$. If there is an isomorphism between a matroid $M$ and $M(A)$, then we also say that $A$ realizes $M$. A matroid is called \emph{$\mathbb{F}$-linear}, or \emph{$\mathbb{F}$-realizable}, if it can be realized by a matrix with entries over $\mathbb{F}$. If a matroid $M$ is $\mathbb{F}$-linear, then so to is $M^*$. A matroid is called \emph{binary} if it is $\mathbb{F}_2${-linear and is called} \emph{regular} if it is $\mathbb{F}$-linear for all fields $\mathbb{F}$. For any field $\mathbb{F}$, the class of $\mathbb{F}$-linear matroids is minor-closed {(Proposition 3.2.4 in} \cite{oxley2006matroid}), thus can be characterized by a (possibly infinite) set of forbidden minors. For example, binary matroids have a simple characterization {in terms of uniform matroids. The \emph{uniform matroid} $U^k_n$ is the matroid on $[n] = \{1,2,\dots,n\}$ whose bases are all $k$-element subsets of $[n]$.}

\begin{thm}[Theorem 6.5.4 in~\cite{oxley2006matroid}] \label{forbid:binary}
    A matroid is binary if and only if it {is} $U^2_4${-free}.
\end{thm}

A \emph{graph} $G = (V,E)$ on $n$ edges consists of a vertex set $V$ and an edge set $E$. Edges have either two vertex endpoints or one, in which case they are called loops. We allow multiple edges to have the same endpoints. We always take $E$ to have a total order. Given a graph $G$, the collection of maximal forests of $G$ give the bases of a matroid on $E$, which we denote as $M(G)$. We say that $G$ represents a matroid $M$ if there is an isomorphism between $M$ and $M(G)$. A \emph{graphic matroid} $M$ is a matroid that can be represented by some graph. Every graphic matroid is regular. 

\subsubsection{Connectivity, direct sums and $2$-sums}

Let $M$ be a matroid and $X \subseteq E(M)$. The \emph{connectivity function of} $M$ {is a map} $\lambda_M{: 2^{E(M)} \to \mathbb{Z}_{\geq 0}}$ defined as
\begin{equation*}
    \lambda_M(X) := r_M(X) + r_M(E \setminus X) - r_M(M).
\end{equation*}
A $k$\emph{-{separation} of} $M$ is a pair $(X,E\setminus X)$ for which $\min\{ \lvert X \rvert, \lvert E \setminus X \rvert \} \geq k$ and $\lambda_M(X) < k$. For $n \geq 2$, a matroid $M$ is $n$\emph{-connected} if, for all $k \in [n-1]$, $M$ has no {$k$-separation}. We call a matroid $M$ \emph{disconnected} if it is not $2$-connected.

Let $M$ and $N$ be matroids on disjoint {ground-}sets. The \emph{direct sum} of $M$ and $N$ is the matroid {$M \oplus N$} whose ground-set is $E(M \oplus N) := E(M) \cup E(N)$ and whose bases are $\mathcal{B}(M \oplus N) := \{ B \cup B' : B \in \mathcal{B}(M), B' \in \mathcal{B}(N) \}$.

A matroid is disconnected if and only if it is the direct sum of two non-empty matroids. Every matroid $M$ decomposes under direct sums into a unique collection of non-empty $2$-connected matroids which we call the \emph{$2$-connected components} (or \emph{connected components}) of $M$. We denote the number of $2$-connected components of a matroid $M$ by $c(M)$.

Let $M$ and $N$ be matroids such that 
    \begin{itemize}
        \item[(i)] $\lvert E(M) \rvert, \lvert E(N) \rvert \geq 2$,
        \item[(ii)] $E(M) \cap E(N) = \{ e \}$, and
        \item[(iii)] neither $(e, E(M) \setminus \{ e\})$ nor $(e, E(N) \setminus \{e\})$ is a $1$-separation of $M$ or $N$ respectively.
    \end{itemize}
Then the \emph{$2$-sum} of $M$ and $N$ is the matroid {$M \oplus_2 N$} on $E(M \oplus_2 N) = (E(M) \cup E(N)) \setminus \{e\}$ whose circuits are 
    \begin{align*}
        \mathcal{C}(M \setminus e)& \cup \mathcal{C}(N \setminus e) \\
        &\cup \{(C_1 \cup C_2) \setminus \{e\} : C_1 \in \mathcal{C}(M), C_2 \in \mathcal{C}(N), e \in C_1 \cap C_2\}.
    \end{align*}

From a matroid {$M \oplus_2 N$}, we can obtain the matroids {$M$} and {$N$} as minors by the following result.

\begin{prop}[Proposition 7.1.21 in~\cite{oxley2006matroid}] \label{prop:2-sum-minor}
    Both $M$ and $N$ are isomorphic to minors of $M \oplus_2 N$. In particular, if {$e'$} is an element of {$E(N) \setminus e$} that is in the same component of $N$ as {$e$}, then there is a minor of $M \oplus_2 N$ from which $M$ can be obtained by relabeling {$e'$} by {$e$}. Moreover, if $N$ has at least three elements, then $M$ isomorphic to a proper minor of $M \oplus_2 N$. 
\end{prop}

\subsubsection{Tree decompositions}

Every matroid $M$ decomposes under direct sums and $2$-sums into a {(not necessarily unique)} collection of non-empty $3$-connected matroids which we call \emph{$3$-connected components} of $M$. Such a decomposition can be represented by a vertex- and edge-labeled tree.

\begin{defi}[\cite{oxley2006matroid}]
    A \emph{matroid-labeled tree} is a tree $T$ with vertex set $\{M_1, M_2, \ldots, M_k\}$ for some positive integer $k$ such that
    \begin{itemize}
        \item[(i)] each $M_i$ is a matroid;
        \item[(ii)] if $M_i$ and $M_j$ are joined by an edge $e$ of $T$, then $E(M_i) \cap E(M_j) = \{e\}$, where $e$ is not a {$1$-separation} of $M_i$ or $M_j$; and
        \item[(iii)] if $M_i$ and $M_j$ are non-adjacent, then $E(M_i) \cap E(M_j) = \emptyset$.
    \end{itemize}
\end{defi}

Let $T$ be a matroid-labeled tree and let {$e$ be an edge of $T$} with endpoints $M_i$ and $M_j$. Then the tree $T/e$, obtained from $T$ by contracting the edge $e$ and labeling with $M_i \oplus_2 M_j$ the new vertex that arises from identifying $M_i$ and $M_j$, is a matroid-labeled tree.

Let $e$ be an edge in some matroid-labeled tree $T$ with endpoints $M_i$ and $M_j$. A \emph{relabeling move} of $e$ consists of relabeling {the edge $e$ in $T$ by some $f$, as well as relabeling the corresponding element $e$ in both $E(M_i)$ and $E(M_j)$ by the same $f$}. We say that two matroid-labeled trees are \emph{equal to within relabeling of their edges} if one can be obtained from the other by a sequence of relabeling moves.

\begin{defi}[\cite{oxley2006matroid}]\label{def:canontree}
    A \emph{tree decomposition} of a $2$-connected matroid $M$ is a matroid-labeled tree $T$ such that if $V(T) = \{M_1, M_2, \ldots, M_k\}$ and {the edge-set of $T$ is} $\{e_1, e_2, \ldots, e_{k-1}\}$, then
    \begin{itemize}
        \item[(i)] $E(M) = (E(M_1) \cup E(M_2) \cup \cdots \cup E(M_k)) \setminus \{e_1,e_2, \ldots, e_{k-1}\}$;
        \item[(ii)] $\lvert E(M_i) \rvert \geq 3$ for all $i$ unless $\lvert E(M) \rvert < 3$, in which case $k=1$ and $M_1 = M$; and
        \item[(iii)] $M$ is the matroid that labels the single vertex of $T/\{e_1,\ldots,e_{k-1}\}$.
    \end{itemize}
\end{defi}

\begin{thm}[Theorem 8.3.10 in~\cite{oxley2006matroid}] \label{thm:canon-tree}
    Let $M$ be a $2$-connected matroid. Then $M$ has a tree decomposition $T$ in which every vertex label is $3$-connected, a circuit, or a cocircuit, and there are no two adjacent vertices that are both labeled by circuits or both labeled by cocircuits. Moreover, $T$ is unique to within relabeling of its edges.
\end{thm}

For a given $2$-connected matroid $M$, we call the tree decomposition given by Theorem~\ref{thm:canon-tree} the \emph{canonical tree decomposition} for $M$. The reader is directed to Section 8.3 in~\cite{oxley2006matroid} for more details on matroid-labeled trees and tree decompositions of $2$-connected matroids.
 
\subsubsection{Series-parallel matroids}

Let $M$ and $N$ be matroids such that $E(M) \cap E(N) = \{e\}$ and $e$ is not a {$1$-separation} of $M$ or $N$. The \emph{series-connection} of $M$ and $N$ is the matroid $S(M,N)$ whose ground-set is $E(S(M,N)) := E(M) \cup E(N)$ and whose set of circuits is
\begin{align*}
    \mathcal{C}(S(M,N)) := \mathcal{C}(M \setminus e) \cup \mathcal{C}(N \setminus e) \cup \{C_1 \cup C_2 : e& \in C_1 \in \mathcal{C}(M),\\ e& \in C_2 \in \mathcal{C}(N)\}.
\end{align*}
The \emph{parallel connection} of $M$ and $N$ is the matroid $P(M,N)$ whose ground-set is $E(P(M,N)) := E(M) \cup E(N)$ and whose set of circuits is
\begin{align*}
    \mathcal{C}(P(M,N)) := \mathcal{C}(M) \cup \mathcal{C}(N) \cup \{(C_1  \cup C_2) \setminus \{e\}) : e& \in C_1 \in \mathcal{C}(M),\\ e& \in C_2 \in \mathcal{C}(N)\}.
\end{align*}
{If $E(M) \cap E(N) = \{e\}$ and $e$ is a $1$-separation of $M$ or $N$, then $e$ is a loop or coloop of $M$ or $N$.}
If $e$ is a loop of $M${,} then {we define the series-parallel connections of $M$ and $N$ as}
\begin{align*}
    P(M,N) &= P(N,M) = M \oplus (N/e), \text{ and}\\
    S(M,N) &= S(N,M) = (M/e) \oplus N.
\end{align*}
If $e$ is a coloop of $M${,} then
\begin{align*}
    P(M,N) &= P(N,M) = (M \setminus e) \oplus N, \text{ and}\\
    S(M,N) &= S(N,M) = M \oplus (N \setminus e).
\end{align*}

We say that $N$ is a \emph{series extension} of $M$ if $N \cong S(M,U^1_2)$ and we say that $N$ is a \emph{parallel extension} of $M$ if $N \cong P(M,U^1_2)$. A matroid $M$ is a \emph{series-parallel matroid} if $M$ can be obtained from $U^0_1$ or $U^1_1$ by a sequence of series and parallel extensions.

\subsubsection{Weak maps}

Let $M$ and $N$ be matroids. A \emph{weak map} from $M$ to $N$ is a bijection $\varphi: E(M) \to E(N)$ such that {for all} $I \in \mathcal{I}(N)$, $\varphi^{-1}[I] \in \mathcal{I}(M)$. We say that a weak map $\varphi: E(M) \to E(N)$ is \emph{rank-preserving} if $r_M(M) = r_N(N)$. A rank-preserving weak map between two matroids implies the following connectivity result.

\begin{prop} \label{n-connect}
    Let $M$ and $N$ be matroids {on $E$} such that $\mathbb{1} : M \to N$ is a rank-preserving weak map. If $N$ is $n$-connected, then $M$ is $n$-connected.
\end{prop}

\begin{proof}
If $\mathbb{1} : M \to N$ is a weak map, then by Proposition 9.1.2(e) in~\cite{white1986theory}, for every $X \subseteq E$, we have $r_M(X) \geq r_N(X)$. If $\mathbb{1} : M \to N$ is rank-preserving, it follows that $\lambda_M(X)\geq \lambda_N(X)$ for all $X \subseteq E$.
\end{proof}

A rank-$k$ matroid $M$ on $n$ elements can be represented by the \emph{matroid base polytope} $\Gamma_M$, which is a subset of the hypersimplex $\Delta_{n,k}$. We can obtain the number of connected components of $M$ from $n$ and the dimension of $\Gamma_M$.

\begin{defi}
    Let $M$ be a matroid on {$[n]$}. For $I \in \mathcal{I}(M)$, let {$e_I = \sum_{i \in I} e_i$}, where $\{e_1,e_2,\ldots,e_n\} \subset \mathbb{R}^n$ are the standard basis vectors. The \emph{matroid base polytope} of $M$ is
    \begin{equation*}
        \Gamma_M = \convex\{e_B : B \in \mathcal{B}(M)\} \subset \mathbb{R}^n.
    \end{equation*}
\end{defi}

\begin{cor}[Corollary 4.7 in~\cite{borovik1997coxeter}] \label{cor:matroid-poly-dim}
    Let $M$ be a matroid on the set {$[n]$} and {let} $c(M)$ {be} the number of connected components of $M$. Then the dimension of the basis matroid polytope associated with $M$ is given by the formula
    \begin{equation*}
        \dim \Gamma_M = n - c(M).
    \end{equation*}
\end{cor}

We can use a rank-preserving weak map $\mathbb{1} : M \to N$ to bound the number of connected components of $M$ above by the number of connected components of $N$.

\begin{cor}
    Let $M$ and $N$ be matroids {on $[n]$}, such that $\mathbb{1} : M \to N$ is a rank preserving weak map. Then, $c(M) \leq c(N)$.
\end{cor}

\begin{proof}
    The rank-preserving weak map $\mathbb{1} : M \to N$ implies that {the preimage of} $\mathcal{B}(N)$ is contained in $\mathcal{B}(M)$. Therefore, the matroid base polytope $\Gamma_N$ is contained in $\Gamma_M$. Thus, the dimension of $\Gamma_N$ is at most the dimension of $\Gamma_M$. Then, by Corollary~\ref{cor:matroid-poly-dim},
    \begin{equation*}
        |E(N)| - c(N) \leq |E(M)| - c(M).
    \end{equation*}
    It follows that $c(M) \leq c(N)$.
\end{proof}

Rank-preserving weak maps interact nicely with matroid minors. Let $M$ be a matroid {and let} $X,Y$ {be} disjoint subsets of $E(M)$. The matroid minor $N = (M/X) \setminus Y$ is called \emph{properly expressed} if it is obtained from $M$ without contracting any loops {or} deleting any coloops (Definition 5.6 in~\cite{lucas1975weak}).

\begin{prop}[Proposition 5.7 in~\cite{lucas1975weak}]
    Any minor can be properly expressed.
\end{prop}

We can use a rank-preserving weak map $\mathbb{1} : M \to N$ and a properly expressed minor $N'$ of $N$ to find a minor $M'$ of $M$ for which $\mathbb{1} : M' \to N'$ is a rank-preserving weak map.

\begin{thm}[Theorem 5.8 in~\cite{lucas1975weak}] \label{thm:minor-contain}
Let $M$ and $N$ be matroids such that {$E(M) = E(N)$ and} $\mathbb{1} : M \to N$ is a rank-preserving weak map. If $X,Y \subseteq E(N)$ and $N'=(N/X)\setminus Y$ is a properly expressed minor of $N$, then $M'=(M/X) \setminus Y$ is a minor of $M$ such that $\mathbb{1} : M' \to N'$ is a rank-preserving weak map.
\end{thm}

For any {rank-$k$} matroid $N$ on $X${,} where {$|X| = n$}, {there exists a matroid $R \cong U^k_n$ such that} $\mathcal{B}(N) \subseteq \mathcal{B}({R})$. This implies the following result.

\begin{cor}[Corollary 5.9 in~\cite{lucas1975weak}]\label{cor:uni-free}
Let $M$ and $N$ be matroids such that $\mathbb{1} : M \to N$ is a rank-preserving weak map. If $M$ is $U^k_n$-free then $N$ is $U^k_n$-free.
\end{cor}

By Theorem~\ref{forbid:binary}, a matroid is binary if and only if it is $U^2_4$-free, so we have the following.

\begin{thm}[Theorem 6.5 in~\cite{lucas1975weak}] \label{thm:class-binary}
    Let $M$ and $N$ be matroid such that $\mathbb{1} : M \to N$ is a rank-preserving weak map. If $M$ in binary then so is $N$. 
\end{thm}

A rank-preserving weak map $\mathbb{1} : M \to N$ is called \emph{non-trivial} if $M \neq N$. If $M$ is binary and $N$ is $2$-connected, then there is no non-trivial rank-preserving weak map $\mathbb{1} : M \to N$.

\begin{thm}[Theorem 6.10 in~\cite{lucas1975weak}] \label{thm:class-binary-nontrivial}
    Let $M$ and $N$ be matroids such that $\mathbb{1} : M \to N$ is a {non-trivial,} rank-preserving weak map. If $M$ is binary, then $N$ is disconnected.
\end{thm}

\subsection{Positroids}\label{sec:positroidbasics}

We now consider the main object of our study, a class of ordered matroids called positroids. An \emph{ordered matroid} $M$ is a matroid whose ground-set $E(M)$ {is totally ordered}. We say that a matroid isomorphism $\phi: M \to N$ is \emph{order preserving} if both $M$ and $N$ are ordered matroids, and {for all} $\{x < y\} \subseteq E(M)$, $\{ \phi(x) < \phi(y) \} \subseteq E(N)$. Recall that the maximal minors of a matrix $A$ is the collection of determinants of all maximal square submatrices of $A$. If $A$ is a square matrix, then the unique maximal minor is the determinant of $A$.

\begin{defi}
    A \emph{positroid} is an $\mathbb{R}$-linear ordered matroid that can be realized by a matrix $A$ over $\mathbb{R}$ with all nonnegative maximal minors.
\end{defi}

{We say that an unordered matroid $M$ has a \emph{positroid ordering} if there exists a positroid $P$ such that $M$ and $P$ are isomorphic as matroids.} The following proposition due to Ardila, Rinc{\'o}n, and Williams shows that positroids are closed under duality and taking minors.

\begin{prop}[Proposition 3.5 in~\cite{ardila2016positroids}] \label{pos-dual-minor}
    Let $M$ be a positroid on {$E = \{e_1 < e_2 < \dots < e_n\}$}. Then $M^*$ is a positroid on {$E$}. Furthermore, for any {$S \subseteq E$}, the deletion $M \setminus S$ and the contraction $M/S$ are both positroids on {$E \setminus S$}. Here the total order on {$E \setminus S$} is the one inherited from {$E$}.
\end{prop}

{Consider a totally ordered set $E = \{e_1 < e_2 < \dots < e_n\}$ and let $i \in [n]$ be fixed}. We define a new total order {$\leq_i$ on $E$} given by 
\begin{equation*}
    {e_i <_i e_{i+1} <_i \dots <_i e_n <_i e_1 <_i e_2 <_i \dots <_i e_{i-1}.}
\end{equation*}
{We take the cyclic interval $[e_i,e_j]$ to be the set of all $e_k \in E$ such that $e_i \leq_i e_k \leq_i e_j$. Let $X = \{x_1 <_i x_2 <_i \dots <_i x_m\}$ and $Y = \{y_1 <_i y_2 <_i \dots <_i y_m\}$ be subsets of $E$. We denote by $X \leq_i Y$ that for all $j \in [m]$, $x_j \leq_i y_j$.
If $X$ is a bounded subset of $E$, then we define $\max_i(X)$ to be the maximum element of $X$ with respect to the total order $\leq_i$. Similarly, we define $\min_i(X)$ to be the minimum element of $X$ with respect to the total order $\leq_i$.} We say that {two subsets $W$ and $Z$ of $E$ are \emph{crossing}} if there exist {$e_w,e_y \in W$} and {$e_x,e_z \in Z$} such that {$e_w <_w e_x <_w e_y <_w e_z$}. If {$W$} and {$Z$} are not crossing subsets of {$E$}, then we call them \emph{non-crossing subsets of {$E$}}. The following characterization of positroids, using non-crossing subsets, is due to Ardila, Rinc{\'o}n, and Williams, building on prior work of da Silva.

\begin{thm}[Theorem 5.1,5.2 in~\cite{ardila2017positively}; Chapter 4, Theorem 1.1 in~\cite{perez1987quelques}] \label{thm:pos-crossing}
    An ordered matroid $M$ on {$E$} is a positroid if and only if for any circuit $C$ and any cocircuit $C^*$ satisfying $C \cap C^* = \emptyset$, the sets $C$ and $C^*$ are non-crossing subsets of {$E$}.
\end{thm}

Let $\mathcal{S} = \{S_1,S_2,\ldots,S_t\}$ be a partition of {a totally ordered set $E$}. Then $\mathcal{S}$ is a \emph{non-crossing partition} if for every pair $i,j \in [t]$, where $i \neq j$, $S_i$ and $S_j$ are non-crossing subsets of {$E$}. Ardila, Rinc{\'o}n, and Williams presented the following characterization of disconnected positroids using non-crossing partitions.

\begin{thm}[Theorem 7.6 in~\cite{ardila2016positroids}] \label{thm:pos-direct-sum}
    Let $M$ be a positroid on {the totally ordered set $E$} and let $S_1, S_2, \dots, S_t$ be the ground-sets of the connected components of $M$. Then $\{S_1, \dots, S_t\}$ is a non-crossing partition of {$E$}. Conversely, if $S_1, S_2, \dots, S_t$ form a non-crossing partition of {$E$} and $M_1, M_2, \dots, M_t$ are connected positroids on $S_1, S_2, \dots, S_t$, respectively, then $M_1 \oplus \dots \oplus M_t$ is a positroid {on $E$}.
\end{thm}

\subsubsection{Dihedral action}

{Let $r$ and $s$ be the bijections on $[n]$ given by, for all $i \in [n]$, $r(i) = i+1$ and $s(i) = n-i+1$ both taken modulo $n$. Take $D_n$ to be the dihedral group on $2n$-elements generated by $r$ and $s$ under function composition. For a totally ordered set $E = \{e_1 < e_2 < \dots < e_n\}$, we define the group action of $D_n$ on $E$ as follows,}
\begin{align*}
    {r \cdot E} &{= \{ e_2 < e_3 < \dots < e_n < e_1 \}}\\
    {s \cdot E} &{= \{ e_n < e_{n-1} < \dots < e_1 \}.}
\end{align*}
{Given an arbitrary subset $X \subseteq E$, for any $\omega \in D_n$, we define the group action of $D_n$ on $X$ as}
\[
    \omega \cdot X = \{e_{w(i)} : e_i \in X\},
\]
{where the ordering is inherited from $\omega \cdot E$. For an ordered matroid $M$ on the ground-set $E$ and for any $\omega \in D_n$, we take $\omega \cdot M$ to be the ordered matroid on $\omega \cdot E$ whose bases are}
\[
    \mathcal{B}(\omega \cdot M) = \{ {\omega \cdot B} : B \in \mathcal{B}(M)\}.
\]

Non-crossing subsets, and therefore positroids, are preserved under dihedral action. {The closure of positroids under dihedral action can also be immediately obtained from Lemma 4.6 in} \cite{blum2001base} {and Corollary 9.5 in} \cite{lam2024polypositroids}.

\subsubsection{Series-parallel connections and 2-sums of positroids}

The following characterization of non-crossing subsets is useful when we consider $2$-sums and series-parallel connections of positroids. Instead of just assuming that the {ground-set} of a matroid {is totally} ordered, it is convenient to assume that {the ground-set is a subset of} a densely {totally} ordered set. Then, it is always possible to add an element to the {ground-set} anywhere in the order.

\begin{obs} \label{crossing-subsets}
    Let $X$ and $Y$ be subsets of a densely totally ordered set {$E$} such that $\lvert X \rvert \geq 2$, $\lvert Y \rvert \geq 2$, and {$X \cap Y = \{e_z\}$}. Then $X$ and $Y$ are non-crossing subsets of {$E$} if and only if
    \begin{itemize}
        \item[(i)] {for all} {$e_x \in X$} and {$e_y \in Y$}, {$e_x \leq_x e_z \leq_x e_y$}; or
        \item[(ii)] {for all} {$e_x \in X$} and {$e_y \in Y$}, {$e_y \leq_y e_z \leq_y e_x$}.
    \end{itemize}
\end{obs}

\begin{proof}
    Without loss of generality, suppose that {for all} {$e_x \in X$} and {$e_y \in Y$}, {$e_x \leq_x e_z \leq_x e_y$}. For distinct {$e_a,e_b \in X$} and {$e_c,e_d \in Y$} we have either {$e_a <_a e_b \leq_a e_c,e_d$} or {$e_b <_b e_a \leq_b e_c,e_d$}. Therefore, $X$ and $Y$ are non-crossing.

    Now, suppose that $X$ and $Y$ are non-crossing subsets of {$E$} and that neither (i) nor (ii) hold. Then there exist {$e_x \in X \setminus \{e_z\}$} and {$e_y \in Y \setminus \{e_z\}$} such that {$e_x <_x e_y <_x e_z$}. Let
    \[ {e_{x'} = \max_x \{ e_x \in X \setminus \{e_z\} : e_x <_x e_y <_x e_z  \}} . \]
    Then there does not exist any {$e_{x''} \in X$} such that {$e_{x'} <_x e_{x''} <_x e_y$}, and since $X$ and $Y$ are non-crossing, there does not exist any {$e_{y'} \in Y$} such that {$e_z <_z e_{y'} <_z e_x$}. {Furthermore,} there does not exist any {$e_{x''} \in X$} such that {$e_y <_x e_{x''} <_x e_z$}. However this implies (ii), {which is} a contradiction.
\end{proof}

It is shown in~\cite{parisi2021m} that series-parallel positroids are closed under series-parallel connections, with appropriate ordering on the elements. {By Corollary 4.18 of} \cite{bonin2024characterization}{, if two matroids have positroid orderings, and their ground-sets have a single element in common, then their series-parallel connection has a positroid ordering.} We extend these results by showing that positroids on non-crossing subsets are closed under series-parallel connections and $2$-sums.

\begin{prop} \label{2sum:pos-construct}
    Let $P$ and $Q$ be positroids on subsets of a densely totally ordered set {$E$} such that $\lvert E(P) \rvert \geq 2$, $\lvert E(Q) \rvert \geq 2$, {$E(P) \cap E(Q) = \{e_z\}$}, and $E(P)$ and $E(Q)$ are non-crossing subsets of {$E$}. Then
    \begin{itemize}
        \item[(i)] $S(P,Q)$ is a positroid.
        \item[(ii)] $P(P,Q)$ is a positroid.
        \item[(iii)] $P \oplus_2 Q$ is a positroid.
    \end{itemize}
\end{prop}

\begin{proof}
    By Observation~\ref{crossing-subsets}, without loss of generality, suppose that {for all} {$e_x \in E(P)$} and {$e_y \in E(Q)$}, {$e_x \leq_x e_z \leq_x e_y$}. 

    By hypothesis $P$ and $Q$ are both positroids, so there exist real{, full-rank} matrices $A_P$ and $A_Q$, with all nonnegative maximal minors, that represent $P$ and $Q$ respectively. As adding a scalar multiple of one row to another row in a square matrix does not change the value of the determinant, we may take $A_P$, $A_Q$ and $A_{S(P,Q)}$ to be as follows
    \begin{equation*}
        A_P {=}
        \left(\begin{array}{@{}c|c@{}}
            A_{P \setminus e_z} &
            \begin{matrix}
                0\\
                \vdots\\
                0\\
                1
            \end{matrix}
        \end{array}\right), \,
        A_Q {=}
        \left(\begin{array}{@{}c|c@{}}
            \begin{matrix}
                1\\
                0\\
                \vdots\\
                0
            \end{matrix} &
            A_{Q \setminus e_z}
        \end{array}\right),\,
        A_{S(P,Q)} {=}
        \left(\begin{array}{@{}c|c|c@{}}
            A_{P \setminus e_z} &
            \begin{matrix}
                0\\
                \vdots\\
                0\\
                1
            \end{matrix} &
            0\\ \hline
            0 &
            \begin{matrix}
                1\\
                0\\
                \vdots\\
                0
            \end{matrix} &
            A_{Q \setminus e_z}
        \end{array}\right).
    \end{equation*}

    The maximal minors of $A_{S(P,Q)}$ are products of maximal minors of $A_P$ and $A_Q$, which are all nonnegative, so all maximal minors of $A_{S(P,Q)}$ are nonnegative. Therefore, $S(P,Q)$ is a positroid. Positroids are closed under taking duals, so $P(P,Q) = [S(P^*,Q^*)]^*$~\cite[Proposition 7.1.14]{oxley2006matroid} is a positroid. As positroids are closed under deletion, {$P(P,Q) \setminus e_z = P \oplus_2 Q$} is a positroid as well.
\end{proof}

\begin{obs}
    Let $P$ and $Q$ be positroids on subsets of a totally ordered set {$E$} such that $\lvert E(P) \rvert = 2$, $\lvert E(Q) \rvert \geq 2$, and {$E(P) \cap E(Q) = \{e_z\}$}. Then $P \oplus_2 Q$ is a positroid if and only if relabeling {$e_z$} in $Q$ by {$E(P) \setminus \{e_z\}$} is an order preserving matroid isomorphism.
\end{obs}

\begin{proof}
    Let {$E(P) = \{e_z,e_x\}$} and $M = P \oplus_2 Q$. Then $M$ can be obtained from $Q$ by relabeling {$e_z$} with {$e_x$}. Thus, $M$ is a positroid if and only if there is an order preserving matroid isomorphism from $M$ to $Q$.
\end{proof}

We have established that the $2$-sum of two non-crossing positroids is a positroid, but every $2$-sum decomposition of a positroid need not be into positroids. However, we show that for every $2$-connected positroid that is not $3$-connected, there exists a $2$-sum decomposition into positroids.

\begin{lemma} \label{2sum:pos-decomp}
    Let {$M = P \oplus_2 Q$} be a $2$-connected positroid on a subset of a densely totally ordered set {$E$} such that {$\lvert E(P) \rvert \geq 3$}, {$\lvert E(Q) \rvert \geq 3$}, and {$E(P) \cap E(Q) = \{e_z\}$}. Then there exist positroids {$P'$} obtained from {$P$} and {$Q'$} obtained from {$Q$} by relabeling {$e_z$} {in $P$ and $Q$ respectively} with some {$e_{z'} \in E$}, such that {$M = P' \oplus_2 Q'$}.
\end{lemma}

\begin{proof}
    Let {$e_x \in E(P) \setminus \{e_z\}$} be fixed and let 
    \[{e_{x'} = \max_x \{ e_{x''} \in E(P) \setminus \{e_z\}: \forall e_y \in E(Q) \setminus \{e_z\}, e_{x''} \leq_x e_y\}.}\] 
    Take {$e_{z'} \in E \setminus \{e_{x'}\}$} such that for all {$e_y \in E(Q) \setminus \{e_z\}$}, {$e_{x'} <_{x'} e_{z'} <_{x'} e_y$}. {We construct} {$P'$} and {$Q'$} from {$P$} and {$Q$} by relabeling {$e_z$} with {$e_{z'}$}, {hence} {$M = P' \oplus_2 Q'$}. Let {$e_{y'} = \min_{x'} \{e_y \in E(Q) \setminus \{e_z\}\}$}. Then, by Proposition~\ref{prop:2-sum-minor}, $M$ contains a minor {$N$} such that {$N$} can be obtained from {$P$} by relabeling {$e_z$} with {$e_{y'}$}. Furthermore, as positroids are closed under taking minors, {$N$} is a positroid. We can obtain {$P'$} from {$N$} by relabeling {$e_{y'}$} with {$e_{z'}$}, and this gives an order preserving isomorphism up to rotation, so {$P'$} is a positroid. By symmetry, we have that {$Q'$} is also a positroid.
\end{proof}

For any $2$-connected positroid $P$, we can take its canonical tree decomposition and appropriately label the edges of the tree $T$ so that any edge-cut of $T$ gives a $2$-sum decomposition of $P$ into non-crossing positroids. To prove this, we first need the following connectivity result.

\begin{prop}[Proposition 8.2.8 in~\cite{oxley2006matroid}] \label{prop:sub-connect}
    If $e$ is an element of an $n$-connected matroid $M$, then, provided $\lvert E(M) \rvert \geq 2(n-1)$, both $M \setminus e$ and $M/e$ are $(n-1)$-connected.
\end{prop}

\begin{prop} \label{2sum:pos-canon-tree}
    Let $M$ be a $2$-connected positroid on a subset of a densely totally ordered set {$E$}. Then $M$ has a canonical tree decomposition $T$ such that each $M' \in V(T)$ is a positroid. Furthermore, for any edge {$e_i$} in $T$, let $T_P$ and $T_Q$ be the connected components of $T \setminus e_i$. Then the corresponding matroids $P$ and $Q$ are two positroids on non-crossing ground-sets.
\end{prop}

\begin{proof}
Let $M$ be a $2$-connected positroid on a subset of a densely totally ordered set {$E$}. By Theorem~\ref{thm:canon-tree}, there exists a canonical tree decomposition $T$ of $M$. Every edge $e_i$ of $T$ represents $M$ as a $2$-sum $P_i \oplus_2 Q_i $. We label each edge $e_i$ according to Lemma~\ref{2sum:pos-decomp}, such that every edge splits $T$ into two positroids.

We claim that in addition, every edge of $T$ now cuts $T$ into two non-crossing ground-sets. {Consider} by way of contradiction that there exists an edge {$e_j$} of $T$ that cuts $T$ into two positroids $P$ and $Q$, such that $E(P)$ and $E(Q)$ are crossing. Then there exist distinct {$e_a,e_b,e_c,e_d \in E(P) \cup E(Q)$} such that {$e_a,e_c \in E(P)$}, {$e_b,e_d \in E(Q)$} and {$e_a <_a e_b <_a e_c <_a e_d$}. Suppose that {$E(P) \cap E(Q) = e_j \in \{e_a,e_b,e_c,e_d\}$}. Then by construction, as shown in the proof of Lemma~\ref{2sum:pos-decomp}, there exists some {$e_z \in E(P) \setminus \{e_j\}$} and {$e_{y'} \in E(Q) \setminus \{e_j\}$}, such that {$e_z,e_j,e_{y'}$} are consecutive in the cyclic ordering of $E(P) \cup E(Q)$. Therefore, {$|\{e_a,e_b,e_c,e_d\}\cap \{ e_z,e_j,e_{y'} \}|\leq 1$}, and {$e_z$}, {$e_j$} or {$e_{y'}$} are interchangeable in any crossing pattern. So, we may restrict ourselves to the case where {$e_j \notin \{e_a,e_b,e_c,e_d\}$}. 

Let $M_P$ and $M_Q$ be the respective vertices of the edge {$e_j$} in $T$. By the properties of canonical tree decompositions (Definition~\ref{def:canontree}), it is not possible for $M_P$ and $M_Q$ to both be circuits or cocircuits. We consider two cases: one of $M_P$ or $M_Q$ is $3$-connected, or one is a circuit and one is a cocircuit. Note that a $2$-connected matroid has the property that any two of its elements lie on a common circuit and on a common cocircuit {(Corollary 4.2.5 and Proposition 4.1.3 in} \cite {oxley2006matroid}).
\begin{itemize}
    \item[(a)] Suppose, without loss of generality, that $M_P$ is $3$-connected and $|E(M_P)|>3$. Then $M_P$ is neither a circuit nor a cocircuit. Then by Proposition~\ref{prop:sub-connect}, {$M_P \setminus e_j$} and {$M_P/e_j$} are both $2$-connected, and therefore {$P \setminus e_j$} and {$P/e_j$} are both $2$-connected. It is known that if a matroid $N$ is $2$-connected and $p \in E(N)$, then $N \setminus p$ or $N/p$ is $2$-connected {(Theorem 4.3.1 in} \cite{oxley2006matroid}). Therefore, {$Q \setminus e_j$} or {$Q/e_j$} is $2$-connected. If {$Q \setminus e_j$} is $2$-connected, then there is a circuit $C$ in {$Q \setminus e_j$} such that {$e_b,e_d\in C$}, and a cocircuit $C^*$ in {$P \setminus e_j$} such that {$e_a,e_c \in C^*$}. Since $C$ and $C^*$ are disjoint, {crossing,} and both lie in $M$, this {contradicts the hypothesis that $M$ is a positroid}. If {$Q/e_j$} is $2$-connected, we similarly find a cocircuit $C^*$ in {$Q/e_j$} and a circuit $C$ in {$P/e_j$}, both of which are in $M$.
    \item[(b)] Suppose, without loss of generality, that $M_P$ is a circuit and $M_Q$ a cocircuit. Then, {$(P/e_j)^* = P^* \setminus e_j$} and {$Q \setminus e_j$} are $2$-connected. Therefore, there is a cocircuit $C^*$ such that
    \[ {e_a,e_b \in C^* \in \mathcal{C}^*(P/e_j) \subseteq \mathcal{C}^*(M)},   \]
    and there is a circuit $C$, disjoint from $C^*$, such that 
    \[ {e_c,e_d  \in C \in \mathcal{C}(Q \setminus e_j) \subseteq \mathcal{C}(M)}.\]
    Since $C$ and $C^*$ are disjoint, {crossing,} and both lie in $M$, this is a contradiction.
\end{itemize}
\end{proof}

Taking Propositions~\ref{2sum:pos-canon-tree} and~\ref{2sum:pos-construct} together we obtain the following characterization of $2$-sums of positroids. 

\begin{thm} \label{pos-canon-tree-iff}
    Let $M$ be a $2$-connected matroid on a subset of a densely totally ordered set {$E$}. Then $M$ is a positroid if and only if there exists a canonical tree decomposition $T$ with the following properties. Every vertex of $T$ is a positroid, and cutting $T$ along any edge {$e_i$} results in two $2$-connected positroids on non-crossing ground-sets.
\end{thm}

\subsection{Related combinatorial objects}\label{sec:decpermbasics}

In~\cite{postnikov2006total}, Postnikov presented {several} classes of combinatorial objects that are in bijection with positroids. We restrict our consideration to {decorated permutations and Grassmann necklaces} as these are useful for studying graphic positroids.

\subsubsection{Decorated permutations}

A \emph{decorated permutation} on {a set $E$} is a pair $\pi^: = (\pi,col)$ consisting of a permutation {$\pi : E \to E$} and a function \[col : \{\mbox{fixed points of }\pi\} \to \{-1,1\}.\]

Informally, a decorated permutation is a bijection on a set that admits two types of fixed points. For a fixed decorated permutation $\pi^: = (\pi,col)$, we define $(\pi^:)^{-1} = (\pi^{-1},-col)$. For a fixed point {$e_i$} of $\pi^:$, we denote
\begin{equation*}
    {\pi^:(e_i) =
    \begin{cases}
        \overline{e_i}, & \text{if } col(i) = -1\\
        \underline{e_i}, & \text{if } col(i) = 1.
    \end{cases}}
\end{equation*}

{For a fixed totally ordered set $E$,} Postnikov defines the {map from the set of all ordered matroids on $E$ to the set of all Grassmann necklaces on $E$, and the map from the set of all ordered matroids on $E$ to the set of all decorated permutations on $E$} \cite{postnikov2006total}. The following direct map{, from the set of all ordered matroids on $E$ to the set of all decorated permutations on $E$,} is equivalent to the composition of {the two maps defined by Postnikov.} Let $M$ be {an ordered} matroid on {a totally ordered set $E$} and {$e_i \in E$}, then the decorated permutation associated to $M$ is {denoted} $\pi^:_M$ and is given by

\begin{equation*}
    {\pi^:_M(e_i) :=
    \begin{cases}
        \overline{e_i}, & \text{if } e_i \text{ is a coloop}\\
        \underline{e_i}, & \text{if } e_i \text{ is a loop}\\
        \min_{i} \left\{ e_j \in E : e_i \in \cl\big([e_{i+1},e_j]\big) \right\}, & \text{otherwise.}
    \end{cases}}
\end{equation*}

This map, {from the set of ordered matroids on $E$ to the set of all decorated permutations on $E$}, interacts nicely with matroid duality, as shown in the following result due to Oh.

\begin{cor}[Corollary 13 in~\cite{oh2009combinatorics}] \label{dec-perm-dual}
For any ordered matroid $M$, we have that $\pi^: _M=(\pi^:_{M^*})^{-1}$.
\end{cor}

We now consider how this map interacts with dihedral action on {$E = \{e_1 < \dots < e_n\}$}. Let $\pi^:$ be a decorated permutation on {$E$} and let {$D_n$} {be the dihedral group consisting of $2n$-elements.} We define the group action of {$D_n$} on the set {of all decorated permutations on $E$} as follows, for all {$\omega \in D_n$ and a fixed $i \in [n]$ such that $\pi^:(e_i) = e_j$,}
\begin{equation*}
    {(\omega \cdot \pi^:)(e_i) :=
    \begin{cases}
        \overline{e_{\omega(i)}}, & \text{if } \pi^:(e_i) = \overline{e_i}\\
        \underline{e_{\omega(i)}}, & \text{if } \pi^:(e_i) = \underline{e_i}\\
        e_{\omega(j)}, & \text{otherwise.}
    \end{cases}}
\end{equation*}

The next remark immediately follows from the definition of Postnikov's map {from the set of all ordered matroids on $E$ to the set of all Grassmann necklaces on $E$} in~\cite{postnikov2006total}.

\begin{rmk} \label{rmk:rotate}
    For any {ordered} matroid $M$ on {$E$}, we have that $r \cdot \pi^:_M = \pi^:_{r \cdot M}$.
\end{rmk}

A reflection of {$E$} interacts with the map {from the set of all ordered matroids on $E$ to the set of all decorated permutations on $E$} in the following way.

\begin{lemma} \label{lem:reflect}
    For any {ordered} matroid $M$ on {$E$} {and a fixed $e_i \in E$ such that $(\pi^:)^{-1}(e_i) = e_k$}, we have that
    \begin{equation*}
        {\pi^:_{s \cdot M}(e_{s(i)}) =
        \begin{cases}
            \overline{e_{s(i)}}, & \text{if } \pi^:_M(e_i) = \overline{e_i}\\
            \underline{e_{s(i)}}, & \text{if } \pi^:_M(e_i) = \underline{e_i}\\
            e_{s(k)}, & \text{otherwise.}
        \end{cases}}
    \end{equation*}
\end{lemma}

\begin{proof}
    Let {$e_i \in E = \{e_1 < \dots < e_n\}$}. Suppose that {$e_i$} is a coloop in $M$, then {$e_{s(i)}$} is a coloop in $s \cdot M$ {hence $\pi^:_{s \cdot M}(e_{s(i)}) = \overline{e_{s(i)}}$}. {Now consider the case where} {$e_i$} is a loop in $M$, then {$e_{s(i)}$} is a loop in $s \cdot M$ {hence $\pi^:_{s \cdot M}(e_{s(i)}) = \underline{e_{s(i)}}$}. {Assume instead that $e_i$} is neither a loop nor a coloop in $M$, then {$e_{s(i)}$} is neither a loop nor a coloop in $s \cdot M$. Furthermore, {$s \cdot E = \{e_n < e_{n-1} < \dots < e_1\},$} so for a fixed $e_i \in E$ such that $(\pi^:_M)^{-1}(e_i) = e_k \in E$, we obtain
    \begin{align*}
        {\pi^:_{s \cdot M}(e_{s(i)})} &= {\min_{s(i)} \left\{ e_j \in s \cdot E : e_{s(i)} \in \cl([e_{s(i-1)},e_j])\right\}}\\ 
        &= {e_{s(k)}}.
    \end{align*}
\end{proof}

When an ordered matroid $M$ is loopless and coloopless, we obtain a nicer characterization of $\pi^:_{s \cdot M}$.

\begin{cor}
    For any loopless and coloopless matroid $M$, we have that $\pi^:_{s \cdot M} = s \cdot (\pi^:_M)^{-1}$.
\end{cor}

Let us now consider the interactions between decorated permutations and the matroid operations of direct sums, series-parallel connections, and $2$-sums. Constructions similar to the ones we present have been given in~\cite{parisi2021m,moerman2021grass}. Ours differ by considering decorated permutations on any finite subset of a totally ordered set, not necessarily $[n]$ for some integer $n$. This is intended to deal with the technical challenges of decompositions of decorated permutations induced by direct sum and $2$-sum decompositions of their corresponding matroids. We begin by defining the following binary operation on decorated permutations.

\begin{defi}
    Let {$E$} be a totally ordered set and let {$\pi^:, \sigma^:$} be decorated permutations on {$X$} and {$Y$} respectively, such that {$X$} and {$Y$} are disjoint subsets of {$E$}. We define the \emph{disjoint union} of {$\pi^:$} and {$\sigma^:$}, which we denote by {$\pi^: \sqcup \sigma^:$}, as the decorated permutation given by
    \begin{equation*}
        {(\pi^: \sqcup \sigma^:)(e_i)} =
        \begin{cases}
            {\pi^:(e_i)}, & \text{if } {e_i \in X}\\
            {\sigma^:(e_i)}, & \text{if } {e_i \in Y}.
        \end{cases}
    \end{equation*}
\end{defi}

\begin{rmk} \label{rmk:dec-perm-direct-sum}
    Let $M$ and $N$ be matroids on disjoint subsets of a totally ordered set {$E$}. Then,
    \begin{equation*}
        \pi^:_{M \oplus N} = \pi^:_M \sqcup \pi^:_N.
    \end{equation*}
\end{rmk}

This characterization of the decorated permutations of direct sums of matroids using disjoint unions is analogous to the work of Moerman and Williams in~\cite[Section 4]{moerman2021grass} and Parisi, Shermann-Bennett, and Williams in~\cite[Section 12]{parisi2021m} using direct sums of permutations. We now turn our attention to the decorated permutations of the $2$-sums of matroids.

\begin{lemma} \label{dec-perm-2sum}
    Let $M$ and $N$ be matroids on subsets of {$E = \{e_1 < e_2 < \dots < e_n\}$} such that $\lvert E(M) \rvert \geq 2$, $\lvert E(N) \rvert \geq 2$, and {$E(M) \cap E(N) = \{e_i\}$}. Suppose that $E(M)$ and $E(N)$ are non-crossing subsets of {$E$}. {If $e_i$} is a loop or coloop of $M$ or $N$, then $\pi^:_{M \oplus_2 N}$ is given by
    \begin{equation*}
        {\pi^:_{M \oplus_2 N} = \pi^:_{M \setminus e_i \oplus N \setminus e_i}}.
    \end{equation*}
    {Otherwise,} $\pi^:_{M \oplus_2 N}$ is given by
    \begin{equation*}
        {\pi^:_{M \oplus_2 N}(e_j)} :=
        \begin{cases}
            {\pi^:_{M}(e_j)}, & \text{if } {e_j \in E(M) \setminus E(N), \pi^:_{M}(e_j) \neq e_i}\\
            {\pi^:_{N}(e_j)}, & \text{if } {e_j \in E(M) \setminus E(N), \pi^:_{M}(e_j) = e_i}\\
            {\pi^:_{N}(e_j)}, & \text{if } {e_j \in E(N) \setminus E(M), \pi^:_{N}(e_j) \neq e_i}\\
            {\pi^:_{M}(e_j)}, & \text{if } {e_j \in E(N) \setminus E(M), \pi^:_{N}(e_j) = e_i}.
        \end{cases}
    \end{equation*}
\end{lemma}

\begin{proof}
    By Observation~\ref{crossing-subsets}, without loss of generality {we may assume} that {for all $e_x \in E(M)$} and {$e_y \in E(N)$, $e_x \leq_x e_i \leq_x e_y$}.
    
    Suppose that {$e_i$} is a loop or coloop in $M$ or $N$, then {$M \oplus_2 N = M \setminus e_i \oplus N \setminus e_i$}. Thus, {$\pi^:_{M \oplus_2 N} = \pi^:_{M \setminus e_i \oplus N \setminus e_i}$}.

    {We now restrict to when $e_i$} is neither a loop nor a coloop of $M$ or $N$. {Consider the case where $e_j \in E(M) \setminus E(N)$} and {$\pi^:_{M}(e_j) \neq e_i$}. Then,
    \begin{align*}
        {\pi^:_{M \oplus_2 N}(e_j)} &= {\min_j \{e_k \in E(M \oplus_2 N) : e_j \in \cl_{M \oplus_2 N}\big([e_{j+1},e_k]\big)\}}\\
        &= {\min_j \{e_k \in E(M \setminus e_i) : e_j \in \cl_M\big([e_{j+1},e_k]\big)\}}\\
        &= {\pi^:_M(e_j)}.
    \end{align*}
    Suppose instead that {$e_j \in E(M) \setminus E(N)$} and {$\pi^:_M(e_j) = e_i$}. Then,
    \begin{align*}
        {\pi^:_{M \oplus_2 N}(e_j)} &= {\min_j \{e_k \in E(M \oplus_2 N) : e_j \in \cl_{M \oplus_2 N}\big([e_{j+1},e_k]\big)\}}\\
        &= {\min_i \{e_k \in E(N \setminus e_i) : e_i \in \cl_N\big([e_{i+1},e_k]\big)\}}\\
        &= {\pi^:_N(e_i)}.
    \end{align*}
The remaining cases follow a similar argument.
\end{proof}

The decorated permutations of $2$-sums can analogously be defined in the language of \emph{amalgamations} of decorated permutations, as shown by Moerman and Williams in~\cite[Section 4]{moerman2021grass}. We complete this section by considering the decorated permutations of series-parallel connections. Parisi, Shermann-Bennett, and Williams have presented similar constructions in~\cite[Section 12]{parisi2021m}.

\begin{lemma}
    Let $M$ and $N$ be $2$-connected matroids on subsets of {$E = \{e_1 < e_2 < \dots < e_n\}$} such that $\lvert E(M) \rvert \geq 2$, $\lvert E(N) \rvert \geq 2$, and {$E(M) \cap E(N) = \{e_i\}$}. Suppose that $E(M)$ and $E(N)$ are non-crossing subsets of {$E$}. Then, {$\pi^:_{P(M,N)}$} is given by
    \begin{equation*}
        {\pi^:_{P(M,N)}(e_j)} :=
        \begin{cases}
            {\pi^:_{M}(e_j)}, & \text{if } {e_j \in E(M) \setminus E(N)}\\
            {\pi^:_{N}(e_i)}, & \text{if } {e_j = e_i}\\
            {\pi^:_{N}(e_j)}, & \text{if } {e_j \in E(N) \setminus E(M), \pi^:_{N}(e_j) \neq e_i}\\
            {\pi^:_{M}(e_i)}, & \text{if } {e_j \in E(N) \setminus E(M), \pi^:_{N}(e_j) = e_i},
        \end{cases}
    \end{equation*}
    and $\pi^:_{S(M,N)}$ is given by 
    \begin{equation*}
        {\pi^:_{S(M,N)}(e_j)} :=
        \begin{cases}
            {\pi^:_{M}(e_j)}, & \text{if } {e_j \in E(M), \pi^:_{M}(e_j) \neq e_i}\\
            {\pi^:_{N}(e_i)}, & \text{if } {e_j \in E(M), \pi^:_{M}(e_j) = e_i}\\
            {\pi^:_{N}(e_j)}, & \text{if } {e_j \in E(N) \setminus E(M), \pi^:_{N}(e_j) \neq e_i}\\
            {\pi^:_{M}(e_i)}, & \text{if } {e_j \in E(N) \setminus E(M), \pi^:_{N}(e_j) = e_i}.
        \end{cases}
    \end{equation*}
\end{lemma}

\begin{proof} By Observation~\ref{crossing-subsets}, there is an element {$e_r$} such that {$E(M)=[e_r,e_i]$} and {$E(N)=[e_i,e_{r-1}]$ where the indices are taken modulo $n$}. We assume $r=1$ without loss of generality. 

For {$r \leq_r j <_r i$}, we have
    \begin{align*}
        {\pi^:_{P(M,N)}(e_j)} &= {\min_j \left\{ e_k \in E(P(M,N)) : e_j \in \cl_{P(M,N)}\big([e_{j+1},e_k]\big) \right\}}\\
        &= {\min_j \left\{ e_k \in E(M) : e_j \in \cl_{M}\big([e_{j+1},e_k]\big) \right\}= \pi^:_{M}(e_j)}.
    \end{align*}

In the remainder of this proof, we omit the first equality that is simply the definition of the decorated permutation. We have,
    \begin{align*}
        {\pi^:_{P(M,N)}(e_i)} &= {\min_i \left\{ e_k \in E(N) : e_i \in \cl_{N}\big([e_{i+1},e_k]\big) \right\} = \pi^:_{N}(e_i)},\\
        {\pi^:_{S(M,N)}(e_i)} &= {\min_r \left\{ e_k \in E(M) : e_i \in \cl_M\big([e_r,e_k]\big) \right\} = \pi^:_M(e_i)}.
    \end{align*}
    Now suppose that {$r \leq_r j <_r i$} and {$\pi^:_{M}(e_j) \neq e_i$}. Then,
    \begin{equation*}
        {\pi^:_{S(M,N)}(e_j) = \min_j \left\{ e_k \in E(M) : e_j \in \cl_{M}\big([e_{j+1},e_k]\big) \right\} = \pi^:_{M}(e_j)}.
    \end{equation*}
    {Consider the case where $r \leq_r j <_r i$} and {$\pi^:_{M}(e_j) = e_i$}. Then,
    \begin{equation*}
        {\pi^:_{S(M,N)}(e_j) = \min_i \left\{ e_k \in E(N) : e_i \in \cl_{N}\big([e_{i+1},e_k]\big) \right\} = \pi^:_{N}(e_i)}.
    \end{equation*}
    {Now assume} that {$i <_i j <_i r$} and {$\pi^:_{N}(e_j) \neq e_i$}. Then,
    \begin{align*}
        {\pi^:_{P(M,N)}(e_j)} &= {\min_j \left\{ e_k \in E(N) : e_j \in \cl_{N}\big([e_{j+1},e_k]\big) \right\} = \pi^:_{N}(e_j)},\\
        {\pi^:_{S(M,N)}(e_j)} & = {\min_j \left\{ e_k \in E(N) : e_j \in \cl_{N}\big([e_{j+1},e_k]\big) \right\} = \pi^:_{N}(e_j)}.
    \end{align*}
    Suppose instead that {$i <_i j <_i r$}  and {$\pi^:_{N}(e_j) = e_i$}. Then,
    \begin{align*}
        {\pi^:_{P(M,N)}(e_j)} & = {\min_j \left\{ e_k \in E(M) : e_i \in \cl_{M}\big([e_r,e_k]\big) \right\} = \pi^:_{M}(e_i)},\\
        {\pi^:_{S(M,N)}(e_j)} &= {\min_i \left\{ e_k \in E(M) : e_i \in \cl_{M}\big([e_{i+1},e_k]\big) \right\} = \pi^:_{M}(e_i)}.
    \end{align*}
\end{proof}

\subsubsection{Grassmann necklaces}

\begin{defi}
    A \emph{Grassmann necklace} on {$E = \{e_1 < e_2 < \dots < e_n\}$} is a sequence {$\mathcal{J} = (J_1,\dots,J_n)$} of subsets {$J_i \subseteq E$} given by
    \begin{equation*}
        {J_{i +1} =
        \begin{cases}
            (J_i \setminus \{e_i\}) \cup \{e_j\}, & \text{if } e_i \in J_i\\
            J_{i+1} = J_i, & \text{otherwise,}
        \end{cases}}
    \end{equation*}
    where the indices are taken modulo $n$.
\end{defi}

Postnikov defines in~\cite{postnikov2006total} the map {from the set of ordered matroids on $E$ to the set of all Grassmann necklaces on $E$} as follows. Let $M$ be a matroid on {$E$}, then the Grassmann necklace associated to $M$ is {$\mathcal{J}(M) = (J_1,\dots,J_n)$}, where {$J_i$} is the lexicographically minimal basis with respect to {$\leq_i$}.

Postnikov presents the following bijection {from the set of all Grassmann necklaces on $E$ to the set of all decorated permutations on $E$} in~\cite[Lemma 16.2]{postnikov2006total}. Let $\mathcal{J}$ be a Grassmann necklace on {$E$} and {$e_i \in E$}, then the decorated permutation associated to $\mathcal{J}$ is {denoted} $\pi^:_{\mathcal{J}}$ and is given by
\begin{equation*}
    {\pi^:_{\mathcal{J}}(e_i) =
    \begin{cases}
        e_j, & \text{if } J_{i+1} = (J_i \setminus \{e_i\}) \cup \{e_j\} \text{ and } e_j \neq e_i\\
        \overline{e_i}, & \text{if } J_{i+1} = J_i \text{ and } e_i \in J_i\\
        \underline{e_i}, & \text{if } J_{i+1} = J_i \text{ and } e_i \notin J_i.
    \end{cases}}
\end{equation*}
Let {$\pi^: = (\pi, col)$} be a decorated permutation on {$E$} and {$e_i \in E$}, then the Grassmann necklace associated to $\pi^:$ is {$\mathcal{J}(\pi^:) = (J_1,\dots,J_n)$} and is given by
\begin{equation*}
    {J_i = \left\{ e_j \in E  : e_j <_i \pi^{-1}(e_j) \mbox{ or } \pi^:(e_j) = \overline{e_j} \right\}.}
\end{equation*}

\begin{figure}[h]
    \begin{tikzpicture}
    \tikzset{edge/.style = {->,> = latex'}}
        \node (1) at (0,2){\{ordered matroids on $E$\}};
        \node (3) at (7,2){\{decorated permutations on $E$\}};
        \node (4) at (7,4){\{Grassmann necklaces on $E$\}};

        \draw[edge] (1) to (3);
        \draw[edge] (1) to (4);
        \draw[edge] (3) to (4);
        \draw[edge] (4) to (3);
    \end{tikzpicture}
    \caption{For a totally ordered set $E$, this diagram commutes.}
    \label{fig:commute-triangle}
\end{figure}
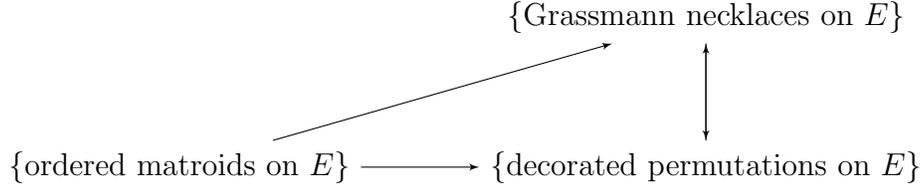

Figure~\ref{fig:commute-triangle} depicts the commutative triangle given by the maps relating {the set of all decorated permutations on $E$, the set of all Grassmann necklaces on $E$, and the set of ordered matroids on $E$.} The commutativity of the diagram implies that for a fixed matroid $M$, the fibers of $\pi^:_M$ and $\mathcal{J}(M)$, under the map {from the set of ordered matroids on $E$ to the set of all decorated permutations on $E$ and the map from the set of ordered matroids on $E$ to the set of all Grassmann necklaces on $E$} respectively, are all equal.

\subsection{Positroid envelopes, envelope classes, and varieties}\label{sec:basics-envelopes}

In this subsection we provide relevant background on positroid envelopes and positroid varieties. We define positroid envelope classes along with binary operations on the positroid envelope classes that are the analogues of direct sums and $2$-sums of positroids. We then show that the positroid envelope classes have decompositions under these direct sum and $2$-sum operations. The reader is directed to~\cite{knutson2013positroid}, wherein Knutson, Lam, and Speyer originally define positroid envelopes, for more information on positroid envelopes and positroid varieties.

\subsubsection{Positroid envelopes}

Oh characterizes the positroids in relation to the fibers of the map {from the set of ordered matroids on $E$ to the set of all Grassmann necklaces on $E$} as {stated in Theorem}~\ref{thm:pos-intersect}. {For a set $X$ we denote by ${X \choose k}$ the set of all $k$-element subsets of $X$.}

\begin{thm}[Theorem 8 in~\cite{oh2011positroids}] \label{thm:pos-intersect}
    Let $P$ be a matroid on {$E = \{e_1 < e_2 < \dots < e_n\}$} of rank $k$ and {$\mathcal{J}(P) = (J_1,\dots,J_n)$} its corresponding Grassmann necklace. Then $P$ is a positroid if and only if
    \begin{equation*}
        \mathcal{B}(P) = \bigcap^n_{{i}=1} \left\{ B {\in} {{E \choose k} : J_i \leq_i B} \right\}.
    \end{equation*}
\end{thm}

This immediately implies that if $P$ is a positroid and $M$ an ordered matroid such that $\mathcal{J}(M) = \mathcal{J}(P)$, then $\mathcal{B}(M) \subseteq \mathcal{B}(P)$. Furthermore, we say that $P$ is the \emph{positroid envelope of $M$}~\cite{knutson2013positroid}. As any matroid $M$ whose positroid envelope is $P$ has the same rank as $P$, we obtain the following.

\begin{cor} \label{cor:pos-env-weak-map}
    Let $M$ be an ordered matroid with positroid envelope $P$. Then $\mathbb{1} : P \to M$ is rank preserving weak map.
\end{cor}

The next two corollaries immediately follow from applying results for rank-preserving weak maps to positroid envelopes.

\begin{cor}
    Let $M$ be an ordered matroid with positroid envelope $P$. Let $N$ be a minor of $M$. Then there exists a minor $Q$ of $P$ such that $\mathbb{1} : Q \to N$ is a rank-preserving weak map.
\end{cor}

\begin{proof}
    This follows from Corollary~\ref{cor:pos-env-weak-map} and Lemma~\ref{thm:minor-contain}.
\end{proof}

\begin{cor}
    Let $M$ be an ordered matroid with positroid envelope $P$. Suppose that $P$ is $U^k_n$-free, then $M$ is $U^k_n$-free.
\end{cor}

\begin{proof}
    This follows from Corollaries~\ref{cor:pos-env-weak-map} and \ref{cor:uni-free}.
\end{proof}

For any rank-$k$ matroid $M$ on $n$ elements, there exists a unique matroid $N$ isomorphic to the uniform matroid $U^k_n$ such that $\mathbb{1} : N \to M$ is a rank-preserving weak map. We use this fact, together with Lemma~\ref{lem:pos-minor-contain}, to provide sufficient conditions {in Corollary}~\ref{cor:pos-uni-minor-contain} for a positroid to contain a minor isomorphic to a uniform matroid.

\begin{lemma} \label{lem:pos-minor-contain}
    Let $M$ be an ordered matroid with positroid envelope $P$. Suppose that $N$ is a minor of $M$, and let $Q$ be the positroid envelope of $N$. Then there exists a minor $K$ of $P$ such that $\mathbb{1}: K \to Q$ is a rank-preserving weak map.
\end{lemma}

\begin{proof}
    Without loss of generality, let $Q$ and $N$ be ordered matroids on {$E = \{e_1 < e_2 < \dots < e_n\}$ where $r_Q(E) = r = r_N(E)$}. By Theorem~\ref{thm:minor-contain}, there exists a minor $K$ of $P$ such that $\mathbb{1} : K \to N$ is a rank-preserving weak map and therefore $\mathcal{B}(N) \subseteq \mathcal{B}(K)$. Let {$\mathcal{J}(Q) = (J_1, \ldots ,J_n) = \mathcal{J}(N)$} and let {$\mathcal{J}(K) = (J'_1, \ldots, J'_n)$}. Then for all {$i \in [n]$, $J'_i \leq_i J_i$}. By Theorem~\ref{thm:pos-intersect},
    \begin{align*}
        {\mathcal{B}(Q)} &{= \bigcap^n_{i=1} \left\{ B \in {E \choose r} : J_i \leq_i B \right\}}\\
        &{\subseteq \bigcap^n_{i=1} \left\{ B \in {E \choose r} : J'_i \leq_i B \right\} = \mathcal{B}(K).} 
    \end{align*}
    Therefore, $\mathbb{1} : K \to Q$ is a rank-preserving weak map.
\end{proof}

\begin{cor} \label{cor:pos-uni-minor-contain}
    Let $M$ be an ordered matroid with positroid envelope $P$. Suppose that $M$ has a minor $N$ with positroid envelope $Q \cong U^k_n$. Then $Q$ is a minor of $P$.
\end{cor}

\begin{proof}
    Let {$E$} be the totally ordered ground-set of $Q$. By Lemma~\ref{lem:pos-minor-contain}, $P$ contains a minor $K$ such that $\mathbb{1}: K \to Q$ is a rank-preserving weak map. Thus,
    \begin{equation*}
        \mathcal{B}(Q) = {{E \choose k}}  \subseteq \mathcal{B}(K) \subseteq {{E \choose k}}.
    \end{equation*}
    Therefore, $K = Q$.
\end{proof}

We now consider the equivalence classes of ordered matroids where two ordered matroids are equivalent if they share the same positroid envelope. We formalize this in the following definition.

\begin{defi}[Positroid envelope class]
    For a fixed positroid $P$, the \emph{envelope class} $\Omega_P$ is the set of all ordered matroids whose positroid envelope is $P$.
\end{defi}

Note that we may equivalently define the positroid envelope classes as follows.
\begin{equation*}
    \Omega_P = \{M : \mathcal{J}(M) = \mathcal{J}(P)\} = \{M : \pi^:_M = \pi^:_P\}.
\end{equation*}

By Corollary~\ref{dec-perm-dual}, we arrive at the following duality result for positroid envelope classes.

\begin{prop} \label{prop:class-dual}
    For any positroid envelope class $\Omega_P$, we have
    \begin{equation*}
        \Omega_{P^*} = \{M^* : M \in \Omega_P\}.
    \end{equation*}
\end{prop}

For a fixed finite, totally ordered set {$E = \{e_1 < e_2 < \dots < e_n\}$}, with {$D_n$ the dihedral group on $2n$-elements}, we define the group action of {$D_n$} on {the set of positroid envelope classes on $E$} as follows. For {$\omega \in D_n$},
\begin{equation*}
    \omega \cdot \Omega_P := \{\omega \cdot M : M \in \Omega_P\}.
\end{equation*}
Dihedral action on positroid envelope classes has the following nice property.

\begin{prop}
    For a positroid envelope class $\Omega_P$ and {$\omega \in D_n$}, we have $\omega \cdot \Omega_P = \Omega_{\omega \cdot P}$.
\end{prop}

\begin{proof}
    Let {$\omega \in D_n$}, then $\omega \cdot \Omega_P = \{\omega \cdot M : \pi^:_M = \pi^:_P\}$. It follows from Remark~\ref{rmk:rotate} and Lemma~\ref{lem:reflect} that 
    \begin{equation*}
        \omega \cdot \Omega_P = \{\omega \cdot M : \pi^:_M = \pi^:_P\} = \{M : \pi^:_{\omega \cdot M} = \pi^:_{\omega \cdot P}\} = \Omega_{\omega \cdot P}.
    \end{equation*}
\end{proof}

If the decorated permutation corresponding to a positroid $P$ consists of a single cycle, then every matroid in $\Omega_P$ is $2$-connected. {We will prove this by using the the following two propositions.}

\begin{prop} [Proposition 4.1.2 in \cite{oxley2006matroid}] \label{prop:equiv-2connect}
    {Define a relation $\xi$ on the ground-set $E(M)$ of a matroid $M$ by $e \, \xi \, f$ if either $e = f$, or $M$ has a circuit containing $\{e,f\}$. For every matroid $M$, the relation $\xi$ is an equivalence relation on $E(M)$.}
\end{prop}

\begin{prop} [Proposition 4.1.3 in \cite{oxley2006matroid}] \label{prop:circ-2connect}
    {A matroid $M$ is $2$-connected if and only if, for every pair of distinct elements of $E(M)$, there is a circuit containing both.}
\end{prop}

\begin{prop} \label{prop:class-2connect}
    Let $P$ be a positroid on {$E = \{e_1 < e_2 < \dots < e_n\}$}. Suppose that $\pi^:_P$ consists of a single cycle. Then {for all} $M \in \Omega_P$, $M$ is $2$-connected.
\end{prop}

\begin{proof}
    Let {$e_i \in E$}, then {$E = \{ e_i, \pi^:(e_i), (\pi^:)^2(e_i), \dots, (\pi^:)^{n-1}(e_i) \}$}. {It follows that for any fixed matroid $M \in \Omega_P$,} {for all} $0 \leq j < n-1$, $(\pi^:)^j(e_i)$ and $(\pi^:(i_k))^{j+1}$ lie on a common circuit. {Thus, by Proposition}~\ref{prop:equiv-2connect}, {any pair of elements of $E$ lie on a common circuit of $M$. Therefore, by Proposition}~\ref{prop:circ-2connect}, $M$ is $2$-connected.
\end{proof}

We can now show that if a decorated permutation contains two cycles on crossing subsets, then the corresponding positroid contains a minor isomorphic to $U^2_4$ and is therefore non-binary.

\begin{lemma} \label{lem:pos-u2_4-minor}
    Let $P$ be a positroid and {$\{e_a < e_b < e_c < e_d\} \subseteq E(P)$}. Suppose that $\pi^:_P$ contains two disjoint cycles $\pi^:_1$ and $\pi^:_2$, on $X_1$ and $X_2$ respectively, such that {$e_a,e_c \in X_1$} and {$e_b,e_d \in X_2$}. Then $P$ contains a minor isomorphic to $U^2_4$.
\end{lemma}

\begin{proof}
    Take {$\pi^:_P = \bigsqcup^n_{i=1} \pi^:_i$}, where {$\{\pi^:_i\}^n_{i=1}$} is the set of disjoint cycles in $\pi^:_P$. Then, for all {$i \in [n]$}, {$\pi^:_i$} is a decorated permutation on {$X_i \subset E(P)$}, where {$\{X_i\}^n_{i=1}$} is a partition of $E(P)$, and there exist positroids {$\{P_i\}^n_{i=1}$} such that {$\pi^:_{P_i} = \pi^:_i$}. Thus, {$M = \bigoplus^n_{i=1} P_i \in \Omega_P$}. By restricting to $E(P_1 \oplus P_2)$, we obtain $P_1 \oplus P_2$ as a minor of $M$, where {$E(P_1) = X_1$}, $E(P_2) = X_2$, $\pi^:_{P_1} = \pi^:_1$, and $\pi^:_{P_2} = \pi^:_2$. As $\pi^:_1$ and $\pi^:_2$ both consist of single cycles, by Proposition~\ref{prop:class-2connect}, both $P_1$ and $P_2$ are $2$-connected matroids. Then, there exist circuits $C_1 \in \mathcal{C}(P_1)$ and $C_2 \in \mathcal{C}(P_2)$ such that {$e_a,e_c \in C_1$} and {$e_b,e_d \in C_2$}. Therefore, {$N_1 = (P_1 \rvert C_1)/(C_1 \setminus \{e_a,e_c\}) \cong U^1_2$} and {$N_2 = (P_2 \rvert C_2)/(C_2 \setminus \{e_b,e_d\}) \cong U^1_2$} are minors of $P_1$ and $P_2$ respectively, so $N_1 \oplus N_2$ is a minor of $P_1 \oplus P_2$. Furthermore, {$\pi^:_{N_1} = (e_a,e_c)$} and {$\pi^:_{N_2} = (e_b,e_d)$}, so by Remark~\ref{rmk:dec-perm-direct-sum}, {$\pi^:_{N_1 \oplus N_2} = (e_a,e_c)(e_b,e_d)$}. Hence, there exists a unique positroid $Q$ such that {$\pi^:_Q = (e_a,e_c)(e_b,e_d)$} and $Q \cong U^2_4$. Then, by Lemma~\ref{lem:pos-minor-contain}, there exists a positroid $K$ that is a minor of $P$ such that $E(K) = E = E(Q)$ and $\mathbb{1} : K \to Q$ {is a rank-preserving weak map}. Thus,
    \[
        {{E \choose 2}} = \mathcal{B}(Q) \subseteq \mathcal{B}(K) \subseteq {{E \choose 2}}.
    \]
    Therefore, $\mathcal{B}(K) \cong U^2_4$.
\end{proof}

We now present some connectivity results for rank-preserving weak maps that we apply to direct sum decompositions of positroid envelope classes in Proposition~\ref{prop:direct-sum-class-decomp}.

\begin{lemma} \label{lem:weak-k-separate}
    Let $M$ and $N$ be matroids {on $E$} such that $\mathbb{1} : M \to N$ is a rank-preserving weak map. Suppose that for $X \subseteq E$, $(X,E\setminus X)$ is a {$k$-separation} of $M$. Then $(X,E)$ is a {$k$-separation} of $N$. 
\end{lemma}

\begin{proof}
    As $\mathbb{1}: M \to N$ is a weak map, $\mathcal{I}(N) \subseteq \mathcal{I}(M)$, hence {for all} $X \in E$, $r_N(X) \leq r_M(X)$. By hypothesis, $(X,E)$ is a {$k$-separation} of $M$, so $\min\{ \lvert X \rvert, \lvert E\setminus X \rvert \} \geq k$ and
    \begin{align*}
        k > \lambda_M(X_1) &= r_M(X_1) + r_M(X_2) - r_M(M)\\
        &\geq r_N(X_1) + r_N(X_2) - r_N(N) = \lambda_N(X_1).
    \end{align*}
    Therefore, $(X,E\setminus X)$ is a {$k$-separation} of $N$.
\end{proof}
It follows that if the image of a rank-preserving weak map is $k$-connected, then the preimage is $k$-connected. Furthermore, if an ordered matroid is $k$-connected, then its positroid envelope is $k$-connected.
\begin{cor}\label{cor:weak-k-connected}
    Let $M$ and $N$ be matroids such that $\mathbb{1} : M \to N$ is a rank-preserving weak map. If $N$ is $k$-connected, then $M$ is $k$-connected.
\end{cor}

\begin{cor}
    Let $M$ be an ordered matroid with positroid envelope $P$. If $M$ is $k$-connected, then $P$ is $k$-connected.
\end{cor}

We define the \emph{direct sum of positroid envelope classes}, a binary operation on positroid envelope classes that is the analogue of direct sums of matroids. We then use this construction to present a direct sum decomposition of positroid envelope classes.

\begin{defi}
Let $M$ and $N$ be matroids on subsets of a totally ordered set $X$, with positroid envelopes $P$ and $Q$ respectively. Suppose that $E(M)$ and $E(N)$ are disjoint. Then we define the \emph{direct sum of envelope classes} $\Omega_P$ and $\Omega_Q$ as
\begin{equation*}
    \Omega_P \oplus \Omega_Q = \{M' \oplus N' : M' \in \Omega_P, N' \in \Omega_Q\}. 
\end{equation*}
\end{defi}

\begin{prop} \label{prop:direct-sum-class-decomp}
    Let $P$ and $Q$ be positroids on disjoint, non-crossing subsets $X_1$ and $X_2$ of a totally ordered set $X$. Then
    \begin{equation*}
        \Omega_{P \oplus Q} = \Omega_P \oplus \Omega_Q.
    \end{equation*}
\end{prop}

\begin{proof}
    {Let $R \in \Omega_P \oplus \Omega_Q$ be arbitrary. Then, there exist $M \in \Omega_P$ and $N \in \Omega_Q$ such that $R = M \oplus N$. By Remark}~\ref{rmk:dec-perm-direct-sum},
    \begin{equation*}
        {\pi^:_R = \pi^:_{M \oplus N} = \pi^:_M \sqcup \pi^:_N = \pi^:_{P \oplus Q}.}
    \end{equation*}
    {Thus, $R \in \Omega_{P \oplus Q}$, and therefore $\Omega_P \oplus \Omega_Q \subseteq \Omega_{P \oplus Q}$.}

    Suppose instead that {$R \in \Omega_{P \oplus Q}$ is arbitrary}. Then, by Lemma~\ref{lem:weak-k-separate}, we have that $(E(P),E(Q))$ is a {$1$-separation} of {$R$}. Let {$M = R|E(P)$} and {$N = R|E(Q)$}, then {$R = M \oplus N$}. Furthermore, {$\pi^:_M = \pi^:_P$} and {$\pi^:_N = \pi^:_Q$}. {Thus, $R = M \oplus N \in \Omega_P \oplus \Omega_Q$, and therefore $\Omega_{P \oplus Q} \subseteq \Omega_P \oplus \Omega_Q$.}
\end{proof}

\subsubsection{Positroid varieties}

For a field $\mathbb{F}$ and nonnegative integers $k < n$, the \emph{Grassmannian} $\Gr(k,\mathbb{F}^n)$ is the collection of all $k$-dimensional linear subspaces of the vector space $\mathbb{F}^n$. Let $V \in \Gr(k,\mathbb{F}^n)$, then $V$ is the row space of a full-rank $k \times n$ matrix $A$. For two distinct full-rank $k \times n$ matrices $A$ and $B$ with entries over $\mathbb{F}$, whose row spaces are both $V$, we have $M(A) = M(B)$. Thus, we may unambiguously define the ordered matroid $M_V = M(A)$ where $A$ is any matrix whose row space is the linear subspace $V$. For a rank-$k$ ordered matroid $M$ on $n$ elements, we define the \emph{matroid stratum} in $\Gr(k,\mathbb{F}^n)$ as
\begin{equation*}
    S_M(\mathbb{F}) := \{V \in \Gr(k,\mathbb{F}^n) : M_V = M\}.
\end{equation*}

The \emph{totally nonnegative Grassmannian} $\Gr^{\geq 0}(k,\mathbb{R}^n)$ consists of the points in the Grassmannian $\Gr(k,\mathbb{R}^n)$ that can be realized by matrices with entries over $\mathbb{R}$ with all nonnegative maximal minors.The \emph{positroid cell} corresponding to the ordered matroid $M$ is defined as
\begin{equation*}
    S_M^{\geq 0}(\mathbb{R}) := S_M(\mathbb{R}) \cap \Gr^{\geq 0}(k,\mathbb{R}^n).
\end{equation*}
The \emph{positroid stratification} of the totally nonnegative {Grassmannian} $\Gr^{\geq 0}(k,\mathbb{R}^n)$ is given by the intersection of the matroid strata with $\Gr(k,\mathbb{R}^n)$, which is exactly the collection of positroid cells. Note that an ordered matroid is a positroid if and only if its corresponding positroid cell is non-empty.

For a positroid $P$, its \emph{open positroid variety} is defined as
\begin{equation*}
    \Pi^{\circ}_{\mathbb{F}}(P) := \{V \in \Gr(k,\mathbb{F}^n) : M_V \in \Omega_P\} = \bigsqcup_{M \in \Omega_P} S_M,
\end{equation*}
and its positroid variety $\Pi_{\mathbb{F}}(P) = \overline{\Pi^{\circ}_{\mathbb{F}}(P)}$.

\begin{cor}[Corollary 5.12 in~\cite{knutson2013positroid}] \label{cor:pos-variety}
    Let $P$ be a positroid. Then, as sets,
    \begin{equation*}
        \Pi_{\mathbb{R}}(P) = \overline{S_P} = \{V \in \Gr(k,\mathbb{R}^n) : I \notin \mathcal{I}(P) \Rightarrow \Delta_I(V) = 0\}.
    \end{equation*}
\end{cor}

\section{Graph constructions}\label{sec:graph-constr}

In this section, we show that for any decorated permutation $\pi^:$, there exists a planar, {edge-ordered} graph $G$ such that $\pi^:=\pi^:_{M(G)}$. Since fixed points of $\pi^:$ are represented by self-loops or leaf-edges anywhere in $G$, and distinct cycles of the permutation may be represented by distinct {connected} components of $G$, the non-trivial part of this proof is in the following lemma, which deals with representing single {$n$-cycles}.

\begin{lemma} \label{lem:matrix-draw-alg}
Let $\pi $ be an {$n$-cycle} on {$E = \{ e_1 < e_2 < \dots < e_n \}$}. Then, there exists a $2$-connected planar, {edge-ordered} graph $G$ such that $\pi=\pi^:_{M(G)}$.
\end{lemma}

\begin{proof}
{Without loss of generality we will assume that $E = [n]$}. We create our graph $G$ as follows. Let
\[ V(G)=\{ v_i : 1 \leq i {< n},\; {\pi^{i-1}(n) > \pi^i(n}) \} {\cup \{v_n\}}, \]
and let the directed, ordered edge-set of $G$ be $E(G) = [n]$.

For each {$1 \leq i < n$},
\begin{itemize}
    \item[(i)] {let $f_R(i) = \min \left\{ k \in [i+1,n] : \pi^k(n) <_n \pi^i(n) \right\}$; and}
    \item[(ii)] {let $f_L(i) = \max \left\{ k \in [1,i] : \pi^i(n) < \pi^{k-1}(n) \right\}$.}
\end{itemize}
{We will build a directed graph for ease of construction, although the final output graph is undirected. For $1 \leq i < n$,} the directed edge {$\pi^i(n)$} goes from {$\pi^i(n)_{out} = v_{f_R(i)}$} to {$\pi^i(n)_{in} = v_{f_L(i)}$}. {The edge $n$ goes from $n_{out} = v_1$ to $n_{in} = v_n$. By construction,}
\[ {\left\{ v_{f_{R(i)}} : 1 \leq i < n \right\} \cup \left\{ v_{f_{L(i)}} : 1 \leq i < n \right\} = V(G).} \]

We embed $G$ in the plane as follows. {Note that $(a,b)$ indicates the entry in the $(a+1)$-th row indexed from the bottom, and the $(b+1)$-th column indexed from the left.} In an {$(n+1)\times (n+1)$} grid, put the label {$n$} in boxes {$(0,n)$} and {$(n,0)$}. Then, for {$1 \leq i <n$}, put the label {$\pi^i(n)$} in the box {$(\pi^i(n),i)$}. {For each vertex $v_i \in V(G)$,} draw $v_i$ {vertically,} stretching from the bottom right corner of {$(\pi^{i-1}(n),i-1)$} to the top left corner of {$(\pi^i(n),i)$}. {For $1 \leq i < n$, draw the edge $\pi^i(n)$} horizontally from {$\pi^i(n)_{out}$} to the bottom right corner of the box {labeled $\pi^i(n)$}, then {diagonally} to its top left corner, and {finally} horizontally to vertex {$\pi^i(n)_{in}$}. The edge {$n$} goes in to the bottom right corner of the box {$(n,0)$}, diagonally to its top left corner, loops around outside of the grid, and comes out of the top left corner of {$(0,n)$}, as shown in Figure~\ref{fig:greedyplanarexample}. By construction, this is a planar embedding of $G$. 

We now embed $G^*$ by putting the label {$n$} in the boxes $(0,0)$ and {$(n,n)$}, and perform the left-right mirror image of the embedding algorithm. The graph $G^*$ has edge set {$E(G^*) = \{ i^* : 1 \leq i \leq n \}$} and vertex set 
\[ V(G^*)= \left\{ v_i^* : 1 < i {\leq n,\; \pi^{i-1}(n) < \pi^i(n)} \right\} {\cup \{ v^*_1 \}} . \]

We first claim that edges of $G$ and $G^*$ only cross when an edge $i$ crosses $i^*$ in the box that is labeled $i$, for {$1\leq i < n$}, and the edge {$n$} crosses {$n^*$} outside of the grid. {Observe} that {$n$} and {$n^*$} have no other crossings.  Note that outside of labeled boxes, edges are horizontal on gridlines. Consider the gridline between two rows {$i$} and {$i+1$}, for some {$1 \leq i < n$}. Without loss of generality, suppose that the box {labeled $i$} is left of {the box labeled $i+1$}. Then all edges of $G$ on the gridline are either left of {the box $i$} or right of {the box $i+1$}, while edges of $G^*$ on the gridline are right of {the box $i$} and left of {the box $i+1$}. Therefore, there are no crossings of edges in $G$ and $G^*$ other than those inside labeled boxes.

We argue that $G^*$ is indeed the dual of $G$, by showing a bijection between vertices of $G^*$ and faces of $G$ (which contain them). Note that by construction every face of $G$ must have at least one vertex of $G^*$ contained in it. Furthermore, if $v^*_i \in \{v^*_1,v^*_n\}$, then $v^*_i$ is the unique vertex of $G^*$ in the face of $G$ that contains it.

{Now,} consider a vertex {$v_i^* \in V(G^*) \setminus \{v^*_1,v^*_n\}$}. {We claim that there exist two internally disjoint, decreasing greedy subsequence of}
\[
    {\pi^{f_L(i)-1}(n),\pi^{f_L(i)}(n),\dots,\pi^{f_R(i-1)}(n)}
\]
{that contain}
\[
    {\left\{\pi^{f_L(i)-1}(n),\pi^i(n),\pi^{f_R(i-1)}(n)\right\}}
\]
{and} 
\[
    {\left\{\pi^{f_L(i)-1}(n),\pi^{i-1}(n),\pi^{f_R(i-1)}(n)\right\}}
\]
{respectively. If we remove $\pi^{f_L(i)-1}(n)$ and $\pi^{f_R(i-1)(n)}$ from both greedy subsequences, then they correspond to two edge-disjoint, reverse-directed paths from $v_{f_R(i-1)}$ to $v_{f_L(i)}$. Therefore, these two paths form a cycle.} Furthermore, for the dual edge $j^*$ in $G^*$ of each edge $j$ on the path {containing $\pi^i(n)$} we have $j_{out}=v_i^*$. Similarly, for the dual edge $j^*$ in $G^*$ of each edge $j$ on the path {containing $\pi^{i-1}(n)$} we have $j_{in}=v_i^*$. {It follows that} the vertex $v_i^*$ uniquely lies inside the face bounded by this cycle.

All that remains to be shown is that {$\pi = \pi_G$}. By duality, we only need to consider any {$\pi^i(n) \mapsto \pi^{i+1}(n)$} such that {$\pi^i(n) > \pi^{i+1}(n)$}. In $G$, the two endpoints of the edge {$\pi^i(n)$} are {$\pi^i(n)_{out} = v_{f_R(i)} = v_{i+1}$} and $\pi^i(n)_{in} = v_{f_L(i)}$. First, we claim that the edge {$\pi^{i+1}(n)$} is the smallest edge with respect to {$\pi^i(n)$} that is incident to {$\pi^i(n)_{out}$}. We {have that $f_L(i+1) = i+1$} and therefore {$\pi^{i+1}(n)_{in} = \pi^i(n)_{out}$}. For any other {$j \neq i+1$} such that {$\pi^j(n)_{in} = \pi^i(n)_{out}$}, we must have {$f_L(j) = i+1$}, and therefore {$\pi^{i+1}(n) < \pi^j(n) < \pi^i(n)$}. For any {$j \neq i$} such that {$\pi^j(n)_{out} = \pi^i(n)_{out}$}, we must have {$f_R(j) = i+1$}, which also implies that {$\pi^{i+1}(n) < \pi^j(n) < \pi^i(n)$}. Therefore, {$\pi_{M(G)}(\pi^i(n)) \geq_{\pi^i(n)} \pi^{i+1}(n)$}. Consider the following cycle. Find a directed path from the edge {$\pi^i(n)$} to {the} edge {$n$} by taking the greedy increasing subsequence of {$\pi^i(n),\pi^{i-1}(n),\dots,n$}. Then, find a reverse directed path from {$\pi^{i+1}(n)$} to {$n$} by taking the greedy decreasing subsequence of {$\pi^{i+1}(n),\pi^{i+2}(n),\dots,n$}. Clearly, the largest edge on this cycle with respect to $\leq_{\pi^i(n)}$ is $\pi^{i+1}(n)$, and therefore {$\pi_{M(G)}(\pi^i(n))=\pi^{i+1}(n)$}. This completes the proof.
\end{proof}

\begin{figure}
    \centering

\begin{tikzpicture}[scale=0.55]
\draw [fill=black!10,draw=white] (0,8) rectangle (1,9);
\draw [fill=black!10,draw=white] (1,3) rectangle (2,4);
\draw [fill=black!10,draw=white] (2,2) rectangle (3,3);
\draw [fill=black!10,draw=white] (3,5) rectangle (4,6);
\draw [fill=black!10,draw=white] (4,4) rectangle (5,5);
\draw [fill=black!10,draw=white] (5,6) rectangle (6,7);
\draw [fill=black!10,draw=white] (6,7) rectangle (7,8);
\draw [fill=black!10,draw=white] (7,1) rectangle (8,2);
\draw [fill=black!10,draw=white] (8,0) rectangle (9,1);

\foreach \a in {0,...,9}{
	\draw [black!10,line width=1pt] (\a,0) -- (\a,9);
}
\foreach \a in {0,...,9}{
	\draw [black!10,line width=1pt] (0,\a) -- (9,\a);
}

\draw [blue!40,line width=3pt,line cap=round] (0,9)  -- (-1,9);
\draw [blue!40,line width=3pt,line cap=round] (-1,-1)  -- (-1,9);
\draw [blue!40,line width=3pt,line cap=round] (-1,-1)  -- (9,-1);
\draw [blue!40,line width=3pt,line cap=round] (9,-1)  -- (9,0);

\draw [blue!40,line width=3pt,line cap=round] (0,9) -- (1,8);
\draw [blue!40,line width=3pt,line cap=round] (3,2) -- (1,4);
\draw [blue!40,line width=3pt,line cap=round] (3,2) -- (7,2);
\draw [blue!40,line width=3pt,line cap=round] (9,0) -- (7,2);
\draw [blue!40,line width=3pt,line cap=round,<-] (1,6) -- (3,6);
\draw [blue!40,line width=3pt,line cap=round] (3,6) -- (5,4);
\draw [blue!40,line width=3pt,line cap=round] (7,4) -- (5,4);
\draw [blue!40,line width=3pt,line cap=round,<-] (1,7) -- (5,7);
\draw [blue!40,line width=3pt,line cap=round] (6,6) -- (5,7);
\draw [blue!40,line width=3pt,line cap=round] (6,6) -- (7,6);
\draw [blue!40,line width=3pt,line cap=round,<-] (1,8) -- (6,8);
\draw [blue!40,line width=3pt,line cap=round] (7,7) -- (6,8);

\draw [fill=black!10,draw=black!10] (.5,8.5) ellipse (.3cm and .3cm);
\draw [fill=black!10,draw=black!10] (8.5,0.5) ellipse (.3cm and .3cm);
\draw [fill=black!10,draw=black!10] (1.5,3.5) ellipse (.3cm and .3cm);
\draw [fill=black!10,draw=black!10] (2.5,2.5) ellipse (.3cm and .3cm);
\draw [fill=black!10,draw=black!10] (3.5,5.5) ellipse (.3cm and .3cm);
\draw [fill=black!10,draw=black!10] (4.5,4.5) ellipse (.3cm and .3cm);
\draw [fill=black!10,draw=black!10] (5.5,6.5) ellipse (.3cm and .3cm);
\draw [fill=black!10,draw=black!10] (6.5,7.5) ellipse (.3cm and .3cm);
\draw [fill=black!10,draw=black!10] (7.5,1.5) ellipse (.3cm and .3cm);

\draw (.5,8.5) node{\Large{8}};
\draw (8.5,.5) node{\Large{8}};
\draw (1.5,3.5) node{\Large{3}};
\draw (2.5,2.5) node{\Large{2}};
\draw (3.5,5.5) node{\Large{5}};
\draw (4.5,4.5) node{\Large{4}};
\draw (5.5,6.5) node{\Large{6}};
\draw (6.5,7.5) node{\Large{7}};
\draw (7.5,1.5) node{\Large{1}};

\draw [fill=black!80,draw=black!80] (1,6) ellipse (.1cm and 2cm);
\draw [fill=black!80,draw=black!80] (2,3) ellipse (.1cm and .1cm);
\draw [fill=black!80,draw=black!80] (4,5) ellipse (.1cm and .1cm);
\draw [fill=black!80,draw=black!80] (7,4.5) ellipse (.1cm and 2.5cm);
\draw [fill=black!80,draw=black!80] (8,1) ellipse (.1cm and .1cm);

\draw (4,-2) node{(a)};

\begin{scope}[xshift=8cm,yshift=3cm,rotate=-45]
    \tikzset{->-/.style={decoration={
  markings,
  mark=at position #1 with {\arrow{>}}},postaction={decorate}}}

    \draw[blue!40,line width=1pt,->-=0.5] (3,5) to (0,7);
    \draw[blue!40,line width=1pt,->-=0.5] (3,3) to [bend left] (0,7);
    \draw[blue!40,line width=1pt,->-=0.5] (6,7) to (0,7);
    \draw[blue!40,line width=1pt,->-=0.5] (6,7) to [bend right] (0,7);
    \draw[blue!40,line width=1pt,->-=0.5] (6,7) to (3,5);
    \draw[blue!40,line width=1pt,->-=0.5] (6,7) to [bend left] (3,3);
    \draw[blue!40,line width=1pt,->-=0.5] (9,7) to (6,7);
    \draw[blue!40,line width=1pt,->-=0.5] (0,7) to [out=-90,in=-90,distance=8cm] (9,7);

    \draw[fill=black] (0,7) circle (4pt);
    \draw[fill=black] (3,5) circle (4pt);
    \draw[fill=black] (3,3) circle (4pt);
    \draw[fill=black] (6,7) circle (4pt);
    \draw[fill=black] (9,7) circle (4pt);

    \draw (1.3,3.6) node{\small{3}};
    \draw (1.1,5.8) node{\small{5}};
    \draw (3,6.5) node{\small{6}};
    \draw (4.5,0.5) node{\small{8}};
    \draw (4.7,3.6) node{\small{2}};
    \draw (4.9,5.8) node{\small{4}};
    \draw (3,8.3) node{\small{7}};
    \draw (7.5,6.5) node{\small{1}};

    \draw (-0.4,7.4) node{\small{$v_1$}};
    \draw (3,4.5) node{\small{$v_4$}};
    \draw (3,2.5) node{\small{$v_2$}};
    \draw (6,7.5) node{\small{$v_7$}};
    \draw (9,7.5) node{\small{$v_8$}};

    \draw (8,1) node{(b)};
\end{scope}

\begin{scope}[yshift=-12cm]

\draw [fill=black!10,draw=white] (0,0) rectangle (1,1);
\draw [fill=black!10,draw=white] (1,3) rectangle (2,4);
\draw [fill=black!10,draw=white] (2,2) rectangle (3,3);
\draw [fill=black!10,draw=white] (3,5) rectangle (4,6);
\draw [fill=black!10,draw=white] (4,4) rectangle (5,5);
\draw [fill=black!10,draw=white] (5,6) rectangle (6,7);
\draw [fill=black!10,draw=white] (6,7) rectangle (7,8);
\draw [fill=black!10,draw=white] (7,1) rectangle (8,2);
\draw [fill=black!10,draw=white] (8,8) rectangle (9,9);

\foreach \a in {0,...,9}{
	\draw [black!10,line width=1pt] (\a,0) -- (\a,9);
}
\foreach \a in {0,...,9}{
	\draw [black!10,line width=1pt] (0,\a) -- (9,\a);
}

\draw [red!50,line width=3pt,dashed] (0,0) -- (0,-1);
\draw [red!50,line width=3pt,dashed] (0,-1) -- (10,-1);
\draw [red!50,line width=3pt,dashed] (10,-1) -- (10,9);
\draw [red!50,line width=3pt,dashed] (9,9) -- (10,9);

\draw [red!50,line width=3pt,dashed] (0,0) -- (1,1);
\draw [red!50,line width=3pt,dashed] (1,3) -- (2,4);
\draw [red!50,line width=3pt,dashed] (3,4) -- (2,4);
\draw [red!50,line width=3pt,dashed] (3,5) -- (4,6);
\draw [red!50,line width=3pt,dashed] (5,6) -- (4,6);
\draw [red!50,line width=3pt,dashed] (5,6) -- (6,7);
\draw [red!50,line width=3pt,dashed] (6,7) -- (7,8);
\draw [red!50,line width=3pt,dashed] (8,8) -- (7,8);
\draw [red!50,line width=3pt,dashed] (8,8) -- (9,9);
\draw [red!50,line width=3pt,dashed] (1,2) -- (2,2);
\draw [red!50,line width=3pt,dashed] (3,3) -- (2,2);
\draw [red!50,line width=3pt,dashed] (3,4) -- (4,4);
\draw [red!50,line width=3pt,dashed] (4,4) -- (5,5);
\draw [red!50,line width=3pt,dashed] (1,1) -- (7,1);
\draw [red!50,line width=3pt,dashed] (8,2) -- (7,1);

\draw [fill=black!10,draw=black!10] (.5,.5) ellipse (.3cm and .3cm);
\draw [fill=black!10,draw=black!10] (8.5,8.5) ellipse (.3cm and .3cm);
\draw [fill=black!10,draw=black!10] (1.5,3.5) ellipse (.3cm and .3cm);
\draw [fill=black!10,draw=black!10] (2.5,2.5) ellipse (.3cm and .3cm);
\draw [fill=black!10,draw=black!10] (3.5,5.5) ellipse (.3cm and .3cm);
\draw [fill=black!10,draw=black!10] (4.5,4.5) ellipse (.3cm and .3cm);
\draw [fill=black!10,draw=black!10] (5.5,6.5) ellipse (.3cm and .3cm);
\draw [fill=black!10,draw=black!10] (6.5,7.5) ellipse (.3cm and .3cm);
\draw [fill=black!10,draw=black!10] (7.5,1.5) ellipse (.3cm and .3cm);

\draw (.5,.5) node{\Large{$8^*$}};
\draw (8.5,8.5) node{\Large{$8^*$}};
\draw (1.5,3.5) node{\Large{$3^*$}};
\draw (2.5,2.5) node{\Large{$2^*$}};
\draw (3.5,5.5) node{\Large{$5^*$}};
\draw (4.5,4.5) node{\Large{$4^*$}};
\draw (5.5,6.5) node{\Large{$6^*$}};
\draw (6.5,7.5) node{\Large{$7^*$}};
\draw (7.5,1.5) node{\Large{$1^*$}};

\draw [fill=black!70,draw=black!70] (1,2) ellipse (.1cm and 1cm);
\draw [fill=black!70,draw=black!70] (3,4) ellipse (.1cm and 1cm);
\draw [fill=black!70,draw=black!70] (5,5.5) ellipse (.1cm and .5cm);
\draw [fill=black!70,draw=black!70] (6,7) ellipse (.1cm and .1cm);
\draw [fill=black!70,draw=black!70] (8,5) ellipse (.1cm and 3cm);

\draw (4,-2) node{(c)};
\end{scope}

\begin{scope}[xshift=14.5cm,yshift=-14cm,rotate=45]

    \tikzset{->-/.style={decoration={
  markings,
  mark=at position #1 with {\arrow{>}}},postaction={decorate}}}
    
    \draw[red!50,line width=1pt,->-=0.5] (0,4) to [bend left] (2.5,4);
    \draw[red!50,line width=1pt,->-=0.5] (0,4) to [bend right] (2.5,4);
    \draw[red!50,line width=1pt,->-=0.5] (2.5,4) to [bend left] (5,4);
    \draw[red!50,line width=1pt,->-=0.5] (2.5,4) to [bend right] (5,4);
    \draw[red!50,line width=1pt,->-=0.5] (5,4) to (7.5,4);
    \draw[red!50,line width=1pt,->-=0.5] (7.5,4) to (10,4);
    \draw[red!50,line width=1pt,->-=0.5] (0,4) to [out=-45,in=-135] (10,4);
     \draw[red!50,line width=1pt,->-=0.5] (0,4) to [out=-90,in=-90,distance=5cm] (10,4);

    \draw[fill=black] (0,4) circle (4pt);
    \draw[fill=black] (2.5,4) circle (4pt);
    \draw[fill=black] (5,4) circle (4pt);
    \draw[fill=black] (7.5,4) circle (4pt);
    \draw[fill=black] (10,4) circle (4pt);

    \draw (1.5,3.3) node{\small{$2^*$}};
    \draw (1.5,4.7) node{\small{$3^*$}};
    \draw (3.7,3.2) node{\small{$4^*$}};
    \draw (3.7,4.8) node{\small{$5^*$}};
    \draw (6.2,4.5) node{\small{$6^*$}};
    \draw (8.7,4.5) node{\small{$7^*$}};
    \draw (5,1.4) node{\small{$1^*$}};
    \draw (5,-0.3) node{\small{$8^*$}};

    \draw (-0.2,4.6) node{\small{$v^*_1$}};
    \draw (2.5,3.5) node{\small{$v^*_3$}};
    \draw (5,3.5) node{\small{$v^*_5$}};
    \draw (7.5,3.5) node{\small{$v^*_6$}};
    \draw (10.2,4.5) node{\small{$v^*_8$}};

    \draw (1.5,-0.5) node{(d)};

\end{scope}
\end{tikzpicture}
    \caption{{(a) Graph $G$, with $\pi_{M(G)} = (8,3,2,5,4,6,7,1)$. (b) Redrawing of $G$. (c) Graph $G^*$, with $\pi_{(M(G))^*} = (\pi_{M(G)})^{-1}$. (d) Redrawing of $G^*$.}}
    \label{fig:greedyplanarexample}
\end{figure}
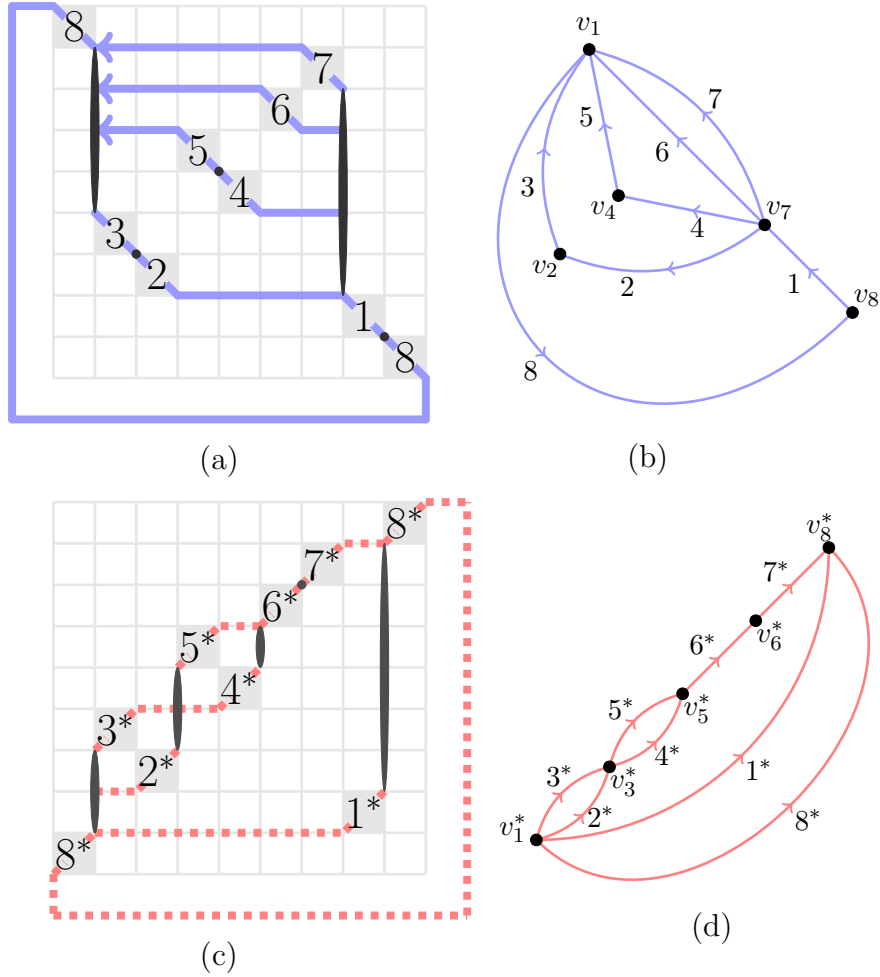

\begin{example}
    {Consider the permutation $\pi = (8,3,2,5,4,6,7,1)$. Figure}~\ref{fig:greedyplanarexample}(a) {depicts the graph $G$ obtained by applying the algorithm from Lemma}~\ref{lem:matrix-draw-alg} {to $\pi$, where $\pi_{M(G)} = (8,3,2,5,4,6,7,1)$. The vertex-set of $G$ is given by}
    \[
        {V(G) = \{v_i : 1 \leq i < 8, \pi^{i-1}(n) > \pi^{i}(n)\} \cup \{v_8\}  = \{v_1,v_2,v_4,v_7,v_8\},}
    \]
    {and the edge-set is $E(G) = \{1,\dots,8\}$. Figure}~\ref{fig:greedyplanarexample}(b) {is a redrawing of the graph $G$. Figure}~\ref{fig:greedyplanarexample}(c) {depicts the graph $G^*$, the dual graph to $G$, obtained by applying the algorithm from Lemma}~\ref{lem:matrix-draw-alg} {to $\pi^{-1}$. We have $\pi_{(M(G))^*} = (\pi_{M(G)})^{-1} = (8,1,7,6,4,5,2,3)$. The vertex-set of $G^*$ is given by}
    \[
        {V(G^*) = \{v^*_i : 1 < i \leq 8, \pi^{i-1}(n) < \pi^i(n)\} \cup \{v^*_1\} = \{v^*_1,v^*_3,v^*_5,v^*_6,v^*_8\},}
    \]
    {and the edge-set is $E(G^*) = \{1^*,\dots,8^*\}$. Figure}~\ref{fig:greedyplanarexample}(d) {is a redrawing of the graph $G^*$.}
\end{example}

{For an arbitrary decorated permutation $\pi^:$, we apply Lemma}~\ref{lem:matrix-draw-alg} {to each disjoint cycle of $\pi^:$ in Theorem}~\ref{thm:graph-build}.

\begin{thm} \label{thm:graph-build}
    Let $\pi^:$ be a decorated permutation. Then, there exists a connected planar graph $G$ such that $\pi^: = \pi^:_{M(G)}$. 
\end{thm}

\begin{proof}
    We write the decorated permutation {$\pi^:$} in disjoint cycle notation, {denoted} $\pi^: = \pi^:_1 \cdots \pi^:_r$, where for all integers $1 \leq i \leq r$, $\pi^:_i$ is a cycle of length $k_i$. {For each $i \in [r]$, we denote the set of elements that $\pi^:_i$ is a permutation of by $S_i$. For a fixed $i \in [r]$, we construct an edge-labeled graph $G_i$ with the property that $\pi^:_{M(G_i)} = \pi^:_i$ as follows.} If {$\pi^:_i$} is a type-$2$ fixed point, then {we take $G_i$ to be the complete graph $K_2$ consisting of two vertices connected by a single edge labeled with the unique element in $S_i$}. Otherwise, we apply Lemma~\ref{lem:matrix-draw-alg} to obtain $G_i$. {It follows that}
    \[
        {\pi^: = \pi^:_1\pi^:_2\dots\pi^:_r = \pi^:_{M(G_1)}\dots\pi^:_{M(G_r)} = \pi^:_{M\left(\bigsqcup G_i\right)}}.
    \]
    Let $\{v_i\}^r_{i=1}$ be a set of vertices such that for each $i \in [r]$, $v_i \in V(G_i)$. Then, $G = \left(\bigsqcup_{i \in [r]} G_i\right)/(v_1 \sim \ldots \sim v_r)$ is connected and Whitney $2$-isomorphic to $\bigsqcup_{i \in [r]} G_i$, so $M(G) \cong M(\bigsqcup_{i \in [r]} G_i)$. Therefore, $\pi^: = \pi^:_{M(G)}$.
\end{proof}

We arrive at the following corollary as an immediate consequence of Theorem~\ref{thm:graph-build}.

\begin{cor}
    Every open positroid variety contains a graphic matroid variety.
\end{cor}

\section{Graphic positroids} \label{sec:graphic-pos}

{In this section, we characterize the graphic positroids as those positroids whose envelope classes consist of exactly one matroid.}
{The following equivalences are known for graphic positroids.}

\begin{cor}[Theorem 5.1 in~\cite{blum2001base}, Corollary 4.18 in~\cite{bonin2024characterization}, Corollary 6.4 in~\cite{speyer2021positive}]\label{cor:pos-sp-u2_4}
    The following are equivalent for a positroid $P$.
    \begin{itemize}
        \item[(i)] $P$ is a graphic.
        \item[(ii)] $P$ is binary.
        \item[(iii)] $P$ is regular.
        \item[(iv)] $P$ is a direct sum of series-parallel matroids.
    \end{itemize}
\end{cor}

We use the following result in Theorem~\ref{thm:pos-graphic} {to extend the characterization of graphic positroids}.

\begin{lemma} \label{lem:dec-perm-graphic}
    Let $P$ be a $2$-connected graphic positroid on a densely totally ordered set {$E$}. Then, $\pi^:_P$ consists of a single cycle.
\end{lemma}

\begin{proof}
    As $P$ is a $2$-connected positroid, by Theorem~\ref{pos-canon-tree-iff}, $P$ has a $2$-sum decomposition into positroids
    \[
        P = P_1 \oplus_2 (P_2 \oplus_2 \cdots (P_{n-1} \oplus_2 P_n) \cdots ),
    \]
    such that for all $i \in [n-1]$, $E(N_i)$ and $E(P_{i+1})$ are non-crossing subsets of {$E$}, where
    \[
        N_i = P_1 \oplus_2 (P_2 \oplus_2 \cdots (P_{i-1} \oplus_2 P_i) \cdots )
    \]
    is a positroid. Furthermore, as $P$ is graphic, for all $i \in [n]$, $P_i$ is a circuit or a cocircuit. We prove by induction that $\pi^:_P$ consists of a single cycle. Let {$E(P_1) = \{j_1 < j_2 < \cdots < j_k\} \subset E$}. Suppose that $P_1$ is a circuit, then $\pi^:_{P_1} = (j_k,j_{k-1},\ldots,j_1)$. Suppose instead that $P_1$ is a cocircuit, then $\pi^:_{P_1} = (j_1,j_2,\ldots,j_k)$. Now, suppose that for some $i \in [n-1]$, $\pi^:_{N_i}$ consists of a single cycle. Then, as $E(N_i)$ and $E(P_{i+1})$ are non-crossing subsets of {$E$}, by Lemma~\ref{dec-perm-2sum}, we have that $\pi^:_{N_i \oplus_2 P_{i+1}}$ consists of a single cycle. It follows by induction that $\pi^:_P = \pi^:_{N_{n-1} \oplus_2 P_n}$ consists of a single cycle.
\end{proof}

\begin{thm}\label{thm:pos-graphic}
    The following are equivalent for a positroid $P$.
    \begin{itemize}
        \item[(i)] $P$ is graphic,
        \item[(ii)] $P$ is the unique matroid contained in $\Omega_P$,
        \item[(iii)] for every field $\mathbb{F}$, $\Pi^{\circ}_{\mathbb{F}}(P) = S_P(\mathbb{F})$,
        \item[(iv)] $\Pi_{\mathbb{R}}(P)$ is equal to the disjoint union of graphic matroid strata,
    \end{itemize} 
\end{thm}

\begin{proof} We break the proof up by items.
\begin{itemize}
    \item[$(i)\Leftrightarrow(ii)$] Suppose that $P$ is a graphic positroid on a densely totally ordered set {$E$}. Then by Theorem~\ref{thm:pos-direct-sum}, $P$ is the direct sum of a collection of $2$-connected graphic positroids, $P = \bigoplus^n_{i=1} P_i$, 
    {whose ground-sets form a non-crossing partition} of {$E$}. By Proposition~\ref{prop:direct-sum-class-decomp}, this induces a direct sum decomposition of the positroid envelope class of $P$, $\Omega_P = \bigoplus^n_{i=1} \Omega_{P_i}$. Thus, $|\Omega_P| = \Pi^n_{i=1} |\Omega_{P_i}|$. It is then sufficient to show that for all $i \in [n]$, $P_i$ is the unique matroid contained in $\Omega_{P_i}$.

    By Lemma~\ref{lem:dec-perm-graphic}, for all $i \in [n]$, $\pi^:_{P_i}$ consists of a single cycle, hence by Proposition~\ref{prop:class-2connect}, all matroids contained in $\Omega_{P_i}$ are $2$-connected. Let $M \in \Omega_{P_i}$. As $P_i$ is binary, {then}
    by Theorem~\ref{thm:class-binary-nontrivial}{,}
    $M = P_i$. Thus, $\Omega_{P_i} = P_i$. It follows that $P$ is the unique matroid contained in $\Omega_P$.

    Now, {consider the case where}
    $P$ is the unique matroid contained in $\Omega_P$. Then by Theorem~\ref{thm:graph-build}, $P$ is graphic.

    \item[$(ii)\Leftrightarrow(iii)$] Let $P$ be a positroid of rank-$k$ on $n$ elements. Suppose that $P$ is the unique matroid contained in $\Omega_P$, hence {$P$ is graphic and therefore regular. So} for any field $\mathbb{F}$, we have
    \begin{align*}
        \Pi^{\circ}_{\mathbb{F}}(P) &= \{V \in \Gr(k,\mathbb{F}^n) : M_V \in \Omega_P\}\\
        &= \{V \in \Gr(k,\mathbb{F}^n) : M_V = P\} = S_P(\mathbb{F}).
    \end{align*}
    {Assume} instead that for every field $\mathbb{F}$, $\Pi^{\circ}_{\mathbb{F}}(P) = S_P(\mathbb{F})$. In particular {$\Pi^{\circ}_{\mathbb{F}_2}(P) = S_P(\mathbb{F}_2)$, so $P$ is binary}. By {Corollary}~\ref{cor:pos-sp-u2_4}, $P$ is graphic and therefore the unique matroid contained in $\Omega_P$.

    \item[$(iii)\Leftrightarrow(iv)$] Let $P$ be a positroid of rank-$k$ on {$[n]$}. Suppose that for every field $\mathbb{F}$, $\Pi^{\circ}_{\mathbb{F}}(P) = S_P(\mathbb{F})$. Then, in particular, $\Pi^{\circ}_{\mathbb{R}} = S_P(\mathbb{R})$. By Corollary~\ref{cor:pos-variety},
    \[
        \Pi_{\mathbb{R}}(P) = \{ V \in \Gr(k,\mathbb{R}^n) : I \notin \mathcal{I}(P) \Rightarrow \Delta_I(V) = 0\}.
    \]
    Observe that $\Pi_{\mathbb{R}}(P)$ can be equivalently defined as
    \[
        \Pi_{\mathbb{R}}(P) = \bigsqcup_{\mathbb{1} : P \xrightarrow{\text{rp}} M_V} S_{M_V},
    \]
    where the disjoint union is taken over all points $V \in \Gr(k,\mathbb{R}^n)$ such that $\mathbb{1} : P \to M_V$ is a rank-preserving weak map. Let $V \in \Gr(k,\mathbb{R}^n)$ be fixed, such that $\mathbb{1} : P \to M_V$ is a rank-preserving weak map. Let {$\mathcal{J}(M_V) = (J_1,J_2,\ldots,J_n)$}, and let $Q$ be the positroid envelope of $M_V$. Then, $\mathcal{J}(Q) = \mathcal{J}(M_V)$. Let {$\mathcal{J}(P) = (J'_1,J'_2,\dots,J'_n)$}, then for all $i \in [n]$, {$J'_i \leq_i J_i$}. By Theorem~\ref{thm:pos-intersect}, we obtain
    \begin{align*}
        \mathcal{B}(Q) &= {\bigcap^n_{i=1} \left\{ B \in {[n] \choose k} : J_i \leq_i B \right\}}\\
        &\subseteq {\bigcap^n_{k=1} \left\{ B \in {[n] \choose k} : J'_i \leq_i B \right\}} = \mathcal{B}(P). 
    \end{align*}
    Therefore, $\mathbb{1} : P \to Q$ is a rank-preserving weak map. As $P$ is the unique matroid contained in $\Omega_P$, {then $P$ is graphic} and therefore binary.
    By Theorem~\ref{thm:class-binary} {and Corollary}~\ref{cor:pos-sp-u2_4}, $Q$ is binary, hence graphic and therefore $M_V = Q$. It follows that $\Pi_{\mathbb{R}}(P)$ is the disjoint union of graphic matroid strata.

    Suppose instead that $\Pi_{\mathbb{R}}(P)$ is the disjoint union of graphic matroid strata. Then, $S_P(\mathbb{F})$ is a graphic matroid stratum, hence $P$ is a graphic and therefore the unique matroid contained in $\Omega_P$. It follows that for every field $\mathbb{F}$, $\Pi^{\circ}_{\mathbb{F}}(P) = S_P(\mathbb{F})$.
\end{itemize}
\end{proof}

\section{Incidence matrices}\label{sec:incidence}

In this section, we strengthen the connection between the graphic positroids and the series-parallel graphs, by showing that every graphic positroid $P$ can be realized by a matrix with all nonnegative maximal minors that is the signed incidence matrix of a graph that represents $P$. Let $A$ be a full rank matrix and let $J$ be a length $r(A)$ sequence of distinct column vectors in $A$, then we denote by $\Delta_J(A)$ the maximal minor of $A$ given by $J$.

\begin{obs} \label{obs:pos-mat-rot}
    Let $B$ be full rank, $m \times n$ matrix with column vectors $(c_1,c_2,\ldots,c_n)$, where $n \geq m$. Suppose that $B$ has all nonnegative maximal minors. Then the full rank, $m \times n$ matrix $B'$, with column vectors 
    \begin{equation*}
        (c_{i+1},\ldots,c_n,(-1)^{m-1}c_1,\ldots,(-1)^{m-1}c_i),
    \end{equation*}
    has all nonnegative maximal minors.
\end{obs}

\begin{proof}
    It is sufficient to show the result for $i = 1$. Let $J$ be a length $m$ sequence of distinct column vectors in $B'$. Suppose that $J \subseteq \{c_2,c_3,\ldots,c_n\}$, then $J$ is also a sequence of distinct column vectors in $B$ and $\Delta_J(B') = \Delta_J(B) \geq 0$. Suppose instead that 
    \[J = (c_{j_1},\ldots,c_{j_{m-1}},(-1)^{m-1}c_1)\] for some $\{c_{j_t}\}^{m-1}_{t=1} \subseteq \{c_2,c_3,\ldots,c_n\}$. Take $J' = (c_1,c_{j_1},\ldots,c_{j_{m-1}})$. Then $\Delta_J(B') = (-1)^{m-1} \cdot (-1)^{m-1}\Delta_{J'}(B) = \Delta_{J'}(B) \geq 0$.
\end{proof}

\begin{thm}\label{prop:spgraph-nonneg}
Let $P$ be a graphic positroid of rank $r$. Then there exists a graph $G$ and matrix $A$, so that $M(G)=M(A)=P$, $A$ has all nonnegative maximal minors and $A$ is obtained from a signed incidence matrix of $G$ by deleting $|G|-r$ rows.
\end{thm}

\begin{proof}
It follows from Theorem~\ref{thm:pos-graphic} that $P$ can be represented by a graph $G$ that is the disjoint union of a set of series-parallel graphs. We show that there exists a full rank matrix $A$, obtained from a signed incidence matrix of $G$ by deleting rows, such that $M(A) = M(G)$ and $A$ has all nonnegative maximal minors. We begin by showing that the result holds for any series-parallel network. Any series-parallel network can be obtained from $C_1$, the cycle graph on a single vertex, or $K_2$, the complete graph on two vertices, by a sequence of series-parallel extensions. Let $G$ be a series-parallel network, and suppose that $G$ can be obtained from $K_2$ by a sequence of series-parallel extensions. Let $K_2 = G_1, G_2, \ldots, G_n = G$ be a sequence of graphs such that for all $1 \leq i < n$, $G_{i+1}$ is obtained from $G_i$ by a series or parallel extension of some edge $e \in E(G_i)$. We inductively construct a sequence of matrices, $A_{G_1}, A_{G_1},\ldots,A_{G_n}$, where {for all} $i \in [n]$, $A_{G_i}$ is a signed incidence matrix for $G_i$. Furthermore, {for all} $i \in [n]$, by deleting the second row we obtain a matrix $A_i$ that has all nonnegative maximal minors. For $G_1 = K_2$, the matrix
    \begin{equation*}
        A_{G_1} =
        \left(\begin{array}{c}
            1\\
            -1
        \end{array}\right)
    \end{equation*}
    is a signed incidence matrix for $K_2$. Let $A_1$ be the matrix obtained from $A_{G_1}$ by deleting the second row. Then $A_1 = (1)$, so it has it exactly one maximal minor, $\det(A) = 1 > 0$.

    Now consider the graph $G_{i+1}$ that is obtained from the graph $G_i$ by a series or parallel extension about some specified edge $e \in E(G_i)$. Let us denote the column vectors of $A_{G_i}$, ordered from left to right, as the sequence $(c_1,c_2,\ldots,c_i)$, and let $c_e$ be the column vector associated with the edge $e \in E(G_i)$. Suppose that $G_{i+1}$ is obtained from $G_i$ by a parallel extension, that is $G_{i+1}$ is obtained from $G_i$ by adding an edge in parallel to $e$ in $G_i$. Suppose further that $A_i$, the matrix obtained from $A_{G_i}$ by deleting the second row, has all nonnegative maximal minors. We take $A_{G_{i+1}}$ to be the matrix whose column vectors, ordered from left to right, is the sequence $(c_1,c_2,\ldots,c_e,c_e,c_{e+1},\ldots,c_i)$. Then $A_{G_{i+1}}$ is a signed incidence matrix of $G_{i+1}$. Furthermore, the set of non-zero maximal minors of $A_{i+1}$, the matrix obtained from $A_{G_{i+1}}$ by deleting the second row, is equal to the set of non-zero maximal minors of $A_i$. Therefore, all maximal minors of $A_{i+1}$ are nonnegative. 

    Suppose instead that $G_{i+1}$ is obtained from $G_i$ by a series extension, that is $G_{i+1}$ is obtained from $G_i$ by adding an edge in series to $e$ in $G_i$. Suppose further that $A_i$, the matrix obtained from $A_{G_i}$ by deleting the second row, has all nonnegative maximal minors. Let the column vector sequence of $A_{G_i}$, ordered from left to right, be given by $(c_1,c_2,\ldots,c_i)$ and let $c_e$ be the column vector corresponding to the edge $e \in E(G_i)$. The matrix $A'_{G_{i+1}}$ given by the sequence of column vectors
    \begin{equation*}
        (c_{e+1},c_{e+2},\ldots,c_i,(-1)^{r(A_i)-1}c_1,(-1)^{r(A_i)-1}c_2,\ldots,(-1)^{r(A_i)-1}c_e)
    \end{equation*}
    is an incidence matrix of $G_i$. Furthermore, by Observation~\ref{obs:pos-mat-rot}, the matrix $A'_i$, obtained from $A'_{G_i}$ by deleting the second row, has all nonnegative maximal minors. We construct a column vector $c'_e$, where $\lvert c'_e \rvert = \lvert c_e \rvert$, with entries given by
    \begin{equation*}
        c'_e[j] :=
        \begin{cases}
            -1, & \text{if } (-1)^{r(A_i)-1}c_e[j] = 1\\
            0, & \text{otherwise.}
        \end{cases}
    \end{equation*}
    Then the matrix
    \begin{equation*}
        A_{G_{i+1}} =
        \left(\begin{array}{@{}c|c|c@{}}
            A'_{G_i} \setminus (-1)^{r(A_i)-1}c_e & (-1)^{r(A_i)-1}c_e + c'_e & c'_e\\ \hline
            \begin{matrix}
                0 \cdots 0
            \end{matrix} &
            1 & 1
        \end{array}\right)
    \end{equation*}
    is a signed incidence matrix for $G_{i+1}$. Furthermore, the matrix $A'_{i+1}$, obtained from $A_{G_{i+1}}$ by deleting the second row, has all nonnegative maximal minors.
\end{proof}

\section*{Acknowledgments}
We thank Allen Knutson, for first posing the question of whether every positroid envelope class contains a graphic matroid, and for many helpful discussions. We also thank Melissa Sherman-Bennett for helpful comments. We are grateful to three anonymous reviewers for helpful comments which significantly improved the quality of this paper.

\bibliographystyle{plain}
\bibliography{references.bib}

\begin{thebibliography}{10}

\bibitem{ardila2016positroids}
Federico Ardila, Felipe Rinc{\'o}n, and Lauren Williams.
\newblock Positroids and non-crossing partitions.
\newblock {\em Transactions of the American Mathematical Society}, 368(1):337--363, 2016.

\bibitem{ardila2017positively}
Federico Ardila, Felipe Rinc{\'o}n, and Lauren~K Williams.
\newblock Positively oriented matroids are realizable.
\newblock {\em Journal of the European Mathematical Society}, 19(3):815--833, 2017.

\bibitem{arkani2012scattering}
Nima Arkani-Hamed, Jacob~L Bourjaily, Freddy Cachazo, Alexander~B Goncharov, Alexander Postnikov, and Jaroslav Trnka.
\newblock Scattering amplitudes and the positive {G}rassmannian.
\newblock {\em arXiv preprint arXiv:1212.5605}, 2012.

\bibitem{blum2001base}
Stefan Blum.
\newblock Base-sortable matroids and koszulness of semigroup rings.
\newblock {\em European Journal of Combinatorics}, 22(7):937--951, 2001.

\bibitem{bonin2024characterization}
Joseph~E Bonin.
\newblock A characterization of positroids, with applications to amalgams and excluded minors.
\newblock {\em European Journal of Combinatorics}, 122:104040, 2024.

\bibitem{borovik1997coxeter}
Alexandre~V Borovik, Israel~M Gelfand, and Neil White.
\newblock Coxeter matroid polytopes.
\newblock {\em Annals of Combinatorics}, 1:123--134, 1997.

\bibitem{knutson2013positroid}
Allen Knutson, Thomas Lam, and David~E Speyer.
\newblock Positroid varieties: juggling and geometry.
\newblock {\em Compositio Mathematica}, 149(10):1710--1752, 2013.

\bibitem{lam2024polypositroids}
Thomas Lam and Alexander Postnikov.
\newblock Polypositroids.
\newblock In {\em Forum of Mathematics, Sigma}, volume~12, page e42. Cambridge University Press, 2024.

\bibitem{lucas1975weak}
Dean Lucas.
\newblock Weak maps of combinatorial geometries.
\newblock {\em Transactions of the American Mathematical Society}, 206:247--279, 1975.

\bibitem{moerman2021grass}
Robert Moerman and Lauren~K Williams.
\newblock Grass trees and forests: {E}numeration of {G}rassmannian trees and forests, with applications to the momentum amplituhedron.
\newblock {\em arXiv preprint arXiv:2112.02061}, 2021.

\bibitem{oh2009combinatorics}
Suho Oh.
\newblock Combinatorics of positroids.
\newblock {\em Discrete Mathematics \& Theoretical Computer Science}, 2009.

\bibitem{oh2011positroids}
Suho Oh.
\newblock Positroids and {S}chubert matroids.
\newblock {\em Journal of Combinatorial Theory, Series A}, 118(8):2426--2435, 2011.

\bibitem{oxley2006matroid}
James~G Oxley.
\newblock {\em Matroid theory}.
\newblock Oxford University Press, USA, {Second} edition, 2011.

\bibitem{parisi2021m}
Matteo Parisi, Melissa Sherman-Bennett, and Lauren Williams.
\newblock The m= 2 amplituhedron and the hypersimplex: signs, clusters, triangulations, {E}ulerian numbers.
\newblock {\em arXiv preprint arXiv:2104.08254}, 2021.

\bibitem{perez1987quelques}
Ilda Perez Fernandez Da~Silva.
\newblock {\em Quelques propri{\'e}t{\'e}s des matroides orient{\'e}s}.
\newblock PhD thesis, Paris 6, 1987.

\bibitem{postnikov2006total}
Alexander Postnikov.
\newblock Total positivity, {G}rassmannians, and networks.
\newblock 2006.
\newblock \url{https://math.mit.edu/~apost/papers/tpgrass.pdf}.

\bibitem{scott2006grassmannians}
Joshua~S Scott.
\newblock Grassmannians and cluster algebras.
\newblock {\em Proceedings of the London Mathematical Society}, 92(2):345--380, 2006.

\bibitem{speyer2021positive}
David Speyer and Lauren Williams.
\newblock The positive {D}ressian equals the positive tropical {G}rassmannian.
\newblock {\em Transactions of the American Mathematical Society, Series B}, 8(11):330--353, 2021.

\bibitem{white1986theory}
Neil White.
\newblock {\em Theory of matroids}.
\newblock Number~26. Cambridge University Press, 1986.

\end{thebibliography}

\end{document}